\documentclass[11pt,a4paper]{amsart}
\ifx\pdfoutput\undefined\else
\pdfoutput=1 
\fi 
\usepackage{fullpage}
\usepackage[foot]{amsaddr}
\usepackage{microtype}
\usepackage[OT1]{fontenc}
\usepackage{eulervm}
\usepackage[tt=false]{libertine} 
\usepackage{amsmath}
\usepackage{amssymb}
\usepackage{amsthm}
\usepackage{mathtools} 
\usepackage[linesnumbered,boxed,ruled,vlined]{algorithm2e}
\usepackage{algpseudocode}
\usepackage{enumitem}
\usepackage{bm}
\usepackage{bbm}
\usepackage{float}
\usepackage{csquotes}
\usepackage{afterpage}

\usepackage{array}
\newcolumntype{L}[1]{>{\raggedright\let\newline\\\arraybackslash\hspace{0pt}}m{#1}}
\newcolumntype{C}[1]{>{\centering\let\newline\\\arraybackslash\hspace{0pt}}m{#1}}
\newcolumntype{R}[1]{>{\raggedleft\let\newline\\\arraybackslash\hspace{0pt}}m{#1}}

\usepackage[margin=1cm]{caption} 
\usepackage{subfig}

\usepackage[thinlines]{easytable}

\usepackage[bookmarks=true,hypertexnames=false,pagebackref]{hyperref}
\hypersetup{colorlinks=true, citecolor=blue, linkcolor=red, urlcolor=blue}

\usepackage{pgfplots}
\pgfplotsset{compat=1.16}
\usepackage{tikz}
\usetikzlibrary{arrows,arrows.meta,backgrounds,calc,fit,decorations.pathreplacing,decorations.markings,shapes.geometric}

\tikzstyle{internal} = [draw, fill, shape=circle]
\tikzstyle{external} = [shape=circle]
\tikzstyle{square}   = [draw, fill, rectangle]
\tikzstyle{triangle} = [draw, fill, regular polygon, regular polygon sides=3, inner sep=3pt]
\tikzstyle{pentagon} = [draw, fill, regular polygon, regular polygon sides=5, inner sep=2pt, minimum size=14pt]
\tikzset{every fit/.append style=text badly centered}

\usetikzlibrary{positioning,chains,fit,shapes,calc}
\usetikzlibrary{trees}
\usetikzlibrary{decorations.pathreplacing}
\usetikzlibrary{decorations.pathmorphing}
\usetikzlibrary{decorations.markings}
\tikzset{>=latex} 

\usepackage{ifthen}

\usepackage{cleveref}

\usepackage[textsize=tiny]{todonotes}

\usepackage[normalem]{ulem}

\usepackage{mleftright}

\usepackage{cool}
\Style{DSymb={\mathrm d},DShorten=true,IntegrateDifferentialDSymb=\mathrm{d}}

\newcommand{\Ex}{\mathop{\mathbb{{}E}}\nolimits}
\renewcommand{\Pr}{\mathop{\mathrm{Pr}}\nolimits}

\def\*#1{\mathbf{#1}}
\def\+#1{\mathcal{#1}}
\def\-#1{\mathrm{#1}}
\def\=#1{\mathbb{#1}}

\newcommand{\abs}[1]{\ensuremath{\left\vert#1\right\vert}}

\newcommand{\Var}[2]{\ensuremath{\textnormal{Var}_{#1}\left(#2\right)}}

\newcommand{\defeq}{:=}

\renewcommand{\P}{\textnormal{\textbf{P}}}
\newcommand{\numP}{\#{\textnormal{\textbf{P}}}}

\DeclareMathOperator{\height}{ht}

\newcommand{\nearrowblue}{\begingroup\color{blue}\nearrow\endgroup}
\newcommand{\nearrowred}{\begingroup\color{red}\nearrow\endgroup}
\newcommand{\searrowblue}{\begingroup\color{blue}\searrow\endgroup}
\newcommand{\searrowred}{\begingroup\color{red}\searrow\endgroup}

\theoremstyle{plain}
\newtheorem{theorem}{Theorem}

\newtheorem{lemma}[theorem]{Lemma}

\newtheorem{fact}[theorem]{Fact}
\newtheorem{proposition}[theorem]{Proposition}
\newtheorem{corollary}[theorem]{Corollary}
\theoremstyle{definition}

\newtheorem{definition}[theorem]{Definition}
\newtheorem{example}[theorem]{Example}
\newtheorem*{example*}{Example}
\newtheorem*{remark}{Remark}

\crefname{theorem}{Theorem}{Theorems}
\crefname{observation}{Observation}{Observations}
\crefname{claim}{Claim}{Claims}
\crefname{condition}{Condition}{Conditions}
\crefname{algorithm}{Algorithm}{Algorithms}
\crefname{property}{Property}{Properties}
\crefname{example}{Example}{Examples}
\crefname{fact}{Fact}{Facts}
\crefname{lemma}{Lemma}{Lemmas}
\crefname{corollary}{Corollary}{Corollaries}
\crefname{definition}{Definition}{Definitions}
\crefname{remark}{Remark}{Remarks}
\crefname{proposition}{Proposition}{Propositions}
\crefname{equation}{equation}{equations}
\crefname{enumi}{Case}{Case}
\creflabelformat{enumi}{(#2#1#3)}

\makeatletter
\def\prob#1#2#3{\goodbreak\begin{list}{}{\labelwidth\z@ \itemindent-\leftmargin
      \itemsep\z@  \topsep6\p@\@plus6\p@
      \let\makelabel\descriptionlabel}
  \item[\textbf{Name}]#1
  \item[\textbf{Instance}]#2
  \item[\textbf{Output}]#3
  \end{list}}
\makeatother

\makeatletter
\providecommand\@dotsep{5}
\def\listtodoname{Todo list}
\def\listoftodos{\@starttoc{tdo}\listtodoname}
\makeatother

\newboolean{doubleblind}
\setboolean{doubleblind}{false}

\newboolean{arxiv}
\setboolean{arxiv}{false}

\title{Rapid mixing of the flip chain over non-crossing spanning trees}

\ifdoubleblind
\author{Author(s)}
\else

\author{Konrad Anand, Weiming Feng, Graham Freifeld, Heng Guo, Mark Jerrum, Jiaheng Wang}

\address[Mark Jerrum]{School of Mathematical Sciences, Queen Mary University of London, Mile End Road, London, E1 4NS, United Kingdom.}
\address[Weiming Feng]{School of Computing and Data Science, The University of Hong Kong, Pokfulam Road, Hong Kong, China.}
\address[Konrad Anand, Graham Freifeld, Heng Guo]{School of Informatics, University of Edinburgh, Informatics Forum, Edinburgh, EH8 9AB, United Kingdom.}
\address[Jiaheng Wang]{Faculty of Informatics and Data Science, University of Regensburg, Bajuwarenstra{\ss}e 4, 93053 Regensburg, Germany.}
\email{konrad.anand@me.com}
\email{wfeng@hku.hk}
\email{g.freifeld@sms.ed.ac.uk}
\email{hguo@inf.ed.ac.uk}
\email{m.jerrum@qmul.ac.uk}
\email{pw384@hotmail.com}

\thanks{This project has received funding from the European Research Council (ERC) under the European Union's Horizon 2020 research and innovation programme (grant agreement No.~947778). Weiming Feng acknowledges the support of Dr. Max R\"ossler, the Walter Haefner Foundation, and the ETH Z\"urich Foundation during his affiliation at ETH Z\"urich. Jiaheng Wang also acknowledges support from the ERC (grant agreement No. 101077083).}

\fi

\begin{document}

\begin{abstract}
    We show that the flip chain for non-crossing spanning trees of $n+1$ points in convex position mixes in time $O(n^8\log n)$.
    We use connections between Fuss-Catalan structures to construct a comparison argument with a chain similar to Wilson's lattice path chain (Wilson 2004).
\end{abstract}
\maketitle

\section{Introduction}

The reconfiguration of planar objects is an intensively studied topic in computational geometry. 
Given a set of points, such a planar structure can be a triangulation, a non-crossing perfect matching, or a \emph{non-crossing spanning trees} (NCST), etc. 
In this paper we focus on the last one. 
Given $n+1$ points in a plane, an NCST is a spanning tree such that no two edges intersect or overlap (except on the joint vertex) when all edges are drawn as straight line segments on the plane. 
The reconfiguration of NCSTs is done by \emph{flips}, which remove an edge of the current tree, and then add an edge back so that the resulting graph is still a valid NCST. 
See \Cref{fig:flip-example} for an example. 
Given two NCSTs over the same point set in general position, Avid and Fukuda \cite{AF96} first proved that it takes at most $2n-2$ flips to move from one to the other. 
On the other hand, Hernando, Hurtado, M\'{a}rquez, Mora and Noy \cite{HHMMN99} constructed an example where it takes at least $1.5n-O(1)$ flips, even when the points are \emph{in convex position}. 
These bounds have not been improved until recently \cite{ABBDDKLLTU22,BGNP23,BdMPW23}.
Very recently, Bjerkevik, Kleist, Ueckerdt, and Vogtenhuber \cite{BKUV24} improved the upper bound to $5/3 n-3$ and the lower bound to $14/9 n-O(1)$ when all points are in convex position. 

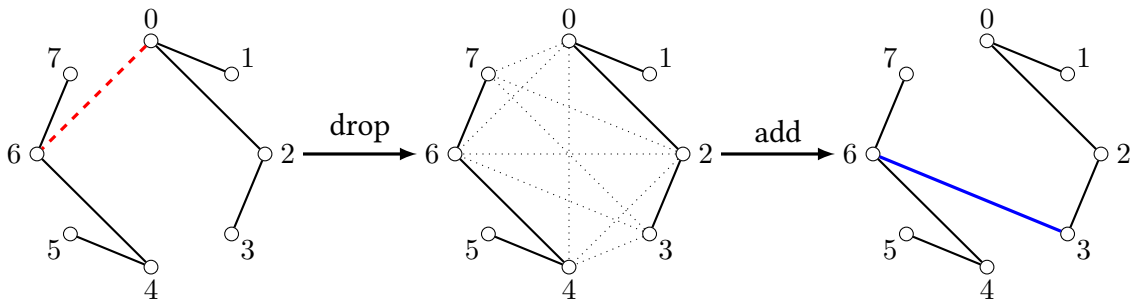
\begin{figure}[ht]
\begin{tikzpicture}
\foreach \x in {0,...,7} {
    \node [shape=circle,draw=black,fill=white,scale=0.5] (A\x) at ({90-360/8*\x}:1.5) {};
    \node [shape=circle,draw=black,fill=white,scale=0.5] (B\x) at ($(A\x)+(5.5,0)$) {};
    \node [shape=circle,draw=black,fill=white,scale=0.5] (C\x) at ($(B\x)+(5.5,0)$) {};
    \node (X\x) at ({90-360/8*\x}:1.8) {$\x$};
    \node (Y\x) at ($(X\x)+(5.5,0)$) {$\x$};
    \node (Z\x) at ($(Y\x)+(5.5,0)$) {$\x$};
}
\path [thick] (A0) edge node [] {} (A1);
\path [thick] (A0) edge node [] {} (A2);
\path [thick] (A2) edge node [] {} (A3);
\path [thick] (A4) edge node [] {} (A5);
\path [thick] (A4) edge node [] {} (A6);
\path [thick] (A6) edge node [] {} (A7);

\path [thick] (B0) edge node [] {} (B1);
\path [thick] (B0) edge node [] {} (B2);
\path [thick] (B2) edge node [] {} (B3);
\path [thick] (B4) edge node [] {} (B5);
\path [thick] (B4) edge node [] {} (B6);
\path [thick] (B6) edge node [] {} (B7);

\path [thick] (C0) edge node [] {} (C1);
\path [thick] (C0) edge node [] {} (C2);
\path [thick] (C2) edge node [] {} (C3);
\path [thick] (C4) edge node [] {} (C5);
\path [thick] (C4) edge node [] {} (C6);
\path [thick] (C6) edge node [] {} (C7);

\path [very thick,dashed,red] (A0) edge node [] {} (A6);
\path [dotted] (B0) edge node [] {} (B4);
\path [dotted] (B0) edge node [] {} (B6);
\path [dotted] (B0) edge node [] {} (B7);
\path [dotted] (B2) edge node [] {} (B4);
\path [dotted] (B2) edge node [] {} (B6);
\path [dotted] (B2) edge node [] {} (B7);
\path [dotted] (B3) edge node [] {} (B4);
\path [dotted] (B3) edge node [] {} (B6);
\path [dotted] (B3) edge node [] {} (B7);
\path [very thick,blue] (C3) edge node [] {} (C6);

\path [->,very thick] (2.0,0) edge node [midway, above] {\large drop} (3.5,0);
\path [->,very thick] (7.5,0) edge node [midway, above] {\large add} (9.0,0);
\end{tikzpicture}
\caption{Illustration of a flip move. The red, dashed edge in the leftmost figure is dropped. This gives $9$ possible edges (dotted in the middle figure) that can be added back to form a valid NCST. The blue, thick edge in the rightmost figure is picked. }
\label{fig:flip-example}
\end{figure}

Another classic topic in computational geometry is to count these planar structures.
Traditional combinatorial studies focus on asymptotic upper and lower bounds of these structures (see e.g.~\cite{MR2011739,SS11,SSW11}).
Motivations for counting these structures have appeared in other topics, such as quantum gravity \cite{JURKIEWICZ1986273, KAZAKOV1985295}, plane curves \cite{MR1036837}, algebraic geometry \cite{MR2075609, MR4186370}, and the theory of disriminants \cite{MR2394437}.
From the computational perspective, the complexity of the counting problem seems to remain elusive in most cases,
with an exception of counting triangulations of (not necessarily simple) polygons known to be \numP-hard \cite{Epp20}.
On the other hand, counting the number of triangulations of a given set of points remains unresolved (but it is conjectured to be \numP-hard, see e.g.~\cite{DRS10} or \cite[\textsection 6.7]{CP23}).
Similarly, counting NCSTs over a set of points in general position apparently is also not known to be \numP-hard nor polynomial-time solvable,
although we conjecture the former to be the case.
To cope with the apparent intractability, sub-exponential time and fixed-parameterised approximation algorithms have been studied \cite{ABCR15,ABRS15,KLS18},
but the complexity of approximate counting NCSTs does not appear to be known either. 

On the other hand, when these $n+1$ points are in convex position, the corresponding counting problem becomes easy to solve. 
One can show that the number of NCSTs, denoted by $C_{2,n}$, is subject to the following recursion: (see e.g.~\cite{Noy98})
\begin{equation} \label{equ:catalan-fuss-2}
  C_{2,0}=1,~ C_{2,n}=\sum_{\substack{i,j,k\in\mathbb{Z}_{\geq 0},\\i+j+k=n-1}}C_{2,i}C_{2,j}C_{2,k},
  \qquad
  \implies
  \qquad
  C_{2,n}=\frac{1}{2n+1}\binom{3n}{n}.
\end{equation}
Sharing a similar form as the famous \emph{Catalan numbers}, this kind of generalisation is called the \emph{$2$-Fuss-Catalan numbers}. 
The ($k$-Fuss-)Catalan numbers capture more than 200 natural combinatorial structures. 
For example, the Catalan numbers count the number of triangulations, non-crossing perfect matchings and non-crossing partitions when the points are in convex position. 
The $k$-Fuss-Catalan numbers count the number of $k$-rooted plane trees, $(k+1)$-ary trees, $k$-Dyck paths, and many more. 
We refer interested readers to Stanley's monograph \cite{Sta15} for an exhaustive survey about these Fuss-Catalan structures.  

In this paper, we are interested in where reconfiguration and approximate counting meet. 
We consider a Markov chain whose steps are random flip moves over NCSTs, denoted the \emph{flip chain}.
In each step, we drop an edge chosen uniformly at random, and then add back an edge uniformly at random among all valid choices.
(This chain is formally defined in \eqref{equ:fm-chain}.)
Reconfiguration results \cite{AF96,HHMMN99,BKUV24} thus imply that the flip chain is \emph{irreducible} (namely, the state space of all NCSTs is connected via flip moves), 
and when the points are in convex position, the \emph{diameter} of the chain is bounded between $14/9n-O(1)$ and $5/3n-3$. 
Markov chains are also the most common approach to approximately count the number of combinatorial structures \cite{JS89,DFK91,JS93}.
While the size of the state space (namely the total number of structures to count) can be exponentially large, 
the hope is that the Markov chain converges to the uniform distribution within polynomially many steps.
The time to converge is called the \emph{mixing time} (defined in \Cref{sec:markov-chain}).
When the mixing time is bounded by a polynomial, the chain is said to be rapid mixing, 
and rapid mixing Markov chains usually lead to efficient approximate counting algorithms.

The mixing times of these geometrically defined Markov chains have received considerable attention from the Markov chain community in the past few decades,
even when restricted to points in convex position,
such as the chains over convex polygon triangulations \cite{MRS97,MT97,EF23}, over triangulations on lattices and spheres \cite{MR3325284, MR3634687}, or over lattice paths and (1-)Dyck paths \cite{MR00,Wil04,CS20}.
See Cohen's thesis \cite{Coh16} for a more detailed explanation and comprehensive overview about random walks over $1$-Fuss-Catalan structures.
Beyond the fields of computational geometry and Markov chain analysis, these chains also find numerous applications, often thanks to the wide connections of Fuss-Catalan numbers, 
such as for quantum spin systems \cite{BCMNS12,MS16,Mov18} or even in algebraic geometry \cite{BDGJPS23}. 


In this paper we give the first polynomial upper bound on the mixing time of the flip chain over NCSTs for points in convex position.

\begin{theorem}  \label{thm:NCST}
For $n+1$ points on the plane in convex position,
the mixing time of the flip chain over their non-crossing spanning trees is $O(n^8\log n)$.
\end{theorem}

The main feature that sets \Cref{thm:NCST} apart from the results mentioned above is that NCSTs for points in convex position are $2$-Fuss-Catalan structures.
This renders some tools unavailable and others, such as the path comparison method \cite{DS93}, a lot more intricate to use.
We also note that while the very similarly defined bases-exchange chain over spanning trees of a graph is known to have an optimal mixing time \cite{ALOV19,CGM19},
NCSTs do not form the bases of a matroid and these results apparently do not apply.
We give an overview of our proofs next.

\subsection{Technical overview}\label{sec:technique}

Our main strategy is to compare the flip chain with another chain that is easier to analyse.
Naturally we would choose a chain over structures that have a bijection with NCSTs.
There are a few candidates (see e.g.~\cite{Sta99}), and we choose $2$-Dyck paths.
The bijection is defined through the same recursive structure, such as \eqref{equ:catalan-fuss-2}, for both $2$-Dyck paths and NCSTs.
Note that while these structures share the same count and have bijections among them,
a local move in one structure may cause drastic changes in another through the bijection.
Thus we need to pick a starting point chain so that the distortion caused by the bijection is manageable and the overhead of the comparison can be bounded by a polynomial.


We represent a $2$-Dyck path by a string consisting of $2n$ ups ($\nearrow$) and $n$ downs ($\searrow$) such that any prefix of the string contains no fewer ups than twice the number of downs. 
The formal definition is given in \Cref{sec:dyck-path}. 
One can verify that the positions of the down arrows form a basis of a matroid.
This generalises the so-called Catalan matroid introduced by Ardila \cite{Ard03}.
Since the bases-exchange walk for these matroids has an $O(n\log n)$ mixing time \cite{CGM19},
a natural choice would be to compare this chain with the flip chain over NCSTs.
This chain randomly chooses a down arrow and move it to a random valid location.
However, as mentioned before, one move in the bases-exchange walk may change NCSTs drastically. 
See \Cref{exp:bad-correspondence}.

\begin{example}\label{exp:bad-correspondence}
Consider two $2$-Dyck paths of length $39$: 

\def\tempmark#1#2#3{
\tikz[overlay]
\draw[fill=#3,fill opacity=0.2,draw opacity=0] (#1,#2) rectangle ++(0.4,-1.15);
}
\resizebox{0.9\linewidth}{!}{
\begin{minipage}{\linewidth}
\begin{equation*}
\tempmark{8.25}{0.7}{orange}
\tempmark{15.2}{0.7}{cyan}
\begin{aligned}
W_1&=\nearrow\nearrow\nearrow\nearrow\nearrow\nearrow\searrow\nearrow\nearrow\nearrow\searrow\searrow\nearrow\nearrow\nearrow\searrow\nearrow\searrow\nearrow\nearrow\nearrow\searrow\nearrow\searrow\nearrow\nearrow\nearrow\nearrow\searrow\nearrow\searrow\nearrow\nearrow\nearrow\searrow\nearrow\searrow\searrow\searrow,\\
W_2&=\nearrow\nearrow\nearrow\nearrow\nearrow\nearrow\searrow\nearrow\nearrow\nearrow\searrow\searrow\nearrow\nearrow\nearrow\searrow\nearrow\searrow\nearrow\searrow\nearrow\searrow\nearrow\searrow\nearrow\nearrow\nearrow\nearrow\searrow\nearrow\searrow\nearrow\nearrow\nearrow\searrow\nearrow\searrow\nearrow\searrow.
\end{aligned}
\end{equation*}
\end{minipage}
}
\undef\tempmark
\vspace{0.1cm}

They differ at position $20$ (shaded by orange) and $38$ (shaded by cyan). 
As $W_2$ can be obtained from $W_1$ through moving the position of one down arrow,
they are connected by one move of the bases-exchange chain.
However, their corresponding NCSTs induced by the bijection of the Fuss-Catalan recurrence (see \Cref{sec:fuss-catalan}) do not share any edge in common, and there is no obvious pattern to note. 
The corresponding trees are illustrated as $W_1$ and $W_2$ in \Cref{fig:correspondence-example}. 
\end{example}

Since the bases-exchange chain appears to be difficult to compare with,
we choose a different chain to control the changes in the corresponding NCSTs,
and that is the adjacent move chain.
Instead of moving a down arrow to a random location, we restrict the movement to be either left or right one position.
More formally, each transition of this chain chooses two adjacent coordinates uniformly at random.
If swapping them results in a valid path, we do so with probability $1/2$,
and otherwise make no change.
For example, suppose the current state is $\nearrow\nearrow\searrow\nearrow\nearrow\searrow$. 
If, with probability $1/5$, we choose the second and third positions,
then the path will not change as swapping them leads to an invalid $2$-Dyck path. 
On the other hand, if (with probability $1/5$) we choose the third and fourth positions,
then the chain will move to $\nearrow\nearrow\nearrow\searrow\nearrow\searrow$ with probability $1/2$ and remain unchanged otherwise.

While this chain is not directly studied before, 
we verify that the coupling method of Wilson \cite{Wil04} still applies here.
Wilson's original argument is applied to Lattice paths or the Bernoulli-Laplacian model (the uniform distribution over \emph{all} strings with a fixed number of ups and downs).
Without too much change, the same argument implies that the adjacent move chain over $2$-Dyck paths mixes in $O(n^3\log n)$ steps.

Our main technical contribution is then to compare the adjacent move chain over $2$-Dyck paths with the flip chain over NCSTs.
We need to characterise how one adjacent move changes the NCSTs.
Consider again $W_2$ in  \Cref{exp:bad-correspondence},
and another $2$-Dyck path $W_3$ that can be reached from $W_2$ by a single adjacent move:

\def\tempmark#1#2#3{
\tikz[overlay]
\draw[fill=#3,fill opacity=0.2,draw opacity=0] (#1,#2) rectangle ++(0.4,-1.15);
}
\resizebox{0.9\linewidth}{!}{
\begin{minipage}{\linewidth}
\begin{equation*}
\tempmark{8.25}{0.7}{orange}
\tempmark{8.65}{0.7}{cyan}
\begin{aligned}
W_2&=\nearrow\nearrow\nearrow\nearrow\nearrow\nearrow\searrow\nearrow\nearrow\nearrow\searrow\searrow\nearrow\nearrow\nearrow\searrow\nearrow\searrow\nearrow\searrow\nearrow\searrow\nearrow\searrow\nearrow\nearrow\nearrow\nearrow\searrow\nearrow\searrow\nearrow\nearrow\nearrow\searrow\nearrow\searrow\nearrow\searrow,\\
W_3&=\nearrow\nearrow\nearrow\nearrow\nearrow\nearrow\searrow\nearrow\nearrow\nearrow\searrow\searrow\nearrow\nearrow\nearrow\searrow\nearrow\searrow\nearrow\nearrow\searrow\searrow\nearrow\searrow\nearrow\nearrow\nearrow\nearrow\searrow\nearrow\searrow\nearrow\nearrow\nearrow\searrow\nearrow\searrow\nearrow\searrow. 
\end{aligned}
\end{equation*}
\end{minipage}
}
\undef\tempmark
\vspace{0.1cm}

The adjacent move swaps the position $20$ and $21$. 
The corresponding NCST of $W_3$ is shown in \Cref{fig:correspondence-example}. 
Note that there are still a lot of altered edges, and in fact, one can also construct examples where $\Omega(n)$ edges get changed after an adjacent move. 
This time, however, the reader may have observed a pattern of the changes: the blue part in the NCST of $W_3$ is the same as the red part in that of $W_2$, but shifted clockwise once. 
As formally proved in \Cref{sec:characterise-adjacent}, each adjacent move corresponds to at most two edge flips, plus one shifting in a particular form (see \eqref{equ:shifting-illustration}). 


\begin{figure}[t]
\begin{tikzpicture}
\foreach \x in {0,...,13} {
    \node [shape=circle,draw=black,fill=white,scale=0.3] (A\x) at ({90-360/14*\x}:2.1) {};
    \node [shape=circle,draw=black,fill=white,scale=0.3] (B\x) at ($(A\x)+(5.4,0)$) {};
    \node [shape=circle,draw=black,fill=white,scale=0.3] (C\x) at ($(B\x)+(5.4,0)$) {};
    \node (X\x) at ({90-360/14*\x}:2.35) {$\scriptstyle\x$};
    \node (Y\x) at ($(X\x)+(5.4,0)$) {$\scriptstyle\x$};
    \node (Z\x) at ($(Y\x)+(5.4,0)$) {$\scriptstyle\x$};
}
\node () at (0,-3) {$W_1$};
\node () at (5.4,-3) {$W_2$};
\node () at (10.8,-3) {$W_3$};

\def\myApath#1#2{\path [thick] (A#1) edge node [] {} (A#2);}
\def\myBpath#1#2{\path [thick] (B#1) edge node [] {} (B#2);}
\def\myCpath#1#2{\path [thick] (C#1) edge node [] {} (C#2);}
\def\myBredpath#1#2{\path [thick,red] (B#1) edge node [] {} (B#2);}
\def\myCbluepath#1#2{\path [thick,blue] (C#1) edge node [] {} (C#2);}
\myApath{0}{13}
\myApath{1}{2}
\myApath{1}{4}
\myApath{1}{13}
\myApath{3}{4}
\myApath{4}{5}
\myApath{4}{6}
\myApath{6}{7}
\myApath{6}{8}
\myApath{9}{10}
\myApath{9}{11}
\myApath{11}{12}
\myApath{11}{13}

\myBredpath{0}{1}
\myBredpath{0}{3}
\myBpath{0}{6}
\myBpath{0}{7}
\myBpath{0}{8}
\myBredpath{2}{3}
\myBredpath{3}{4}
\myBredpath{3}{5}
\myBpath{8}{9}
\myBpath{8}{10}
\myBpath{8}{13}
\myBpath{10}{11}
\myBpath{10}{12}

\myCpath{0}{7}
\myCpath{0}{8}
\myCbluepath{1}{2}
\myCbluepath{1}{4}
\myCbluepath{3}{4}
\myCbluepath{4}{5}
\myCbluepath{4}{6}
\myCpath{6}{7}
\myCpath{8}{9}
\myCpath{8}{10}
\myCpath{8}{13}
\myCpath{10}{11}
\myCpath{10}{12}

\undef\myApath
\undef\myBpath
\undef\myCpath
\undef\myBredpath
\undef\myCbluepath
\end{tikzpicture}
\caption{NCSTs corresponding to the $2$-Dyck paths in \Cref{exp:bad-correspondence}.}
\label{fig:correspondence-example}
\end{figure}
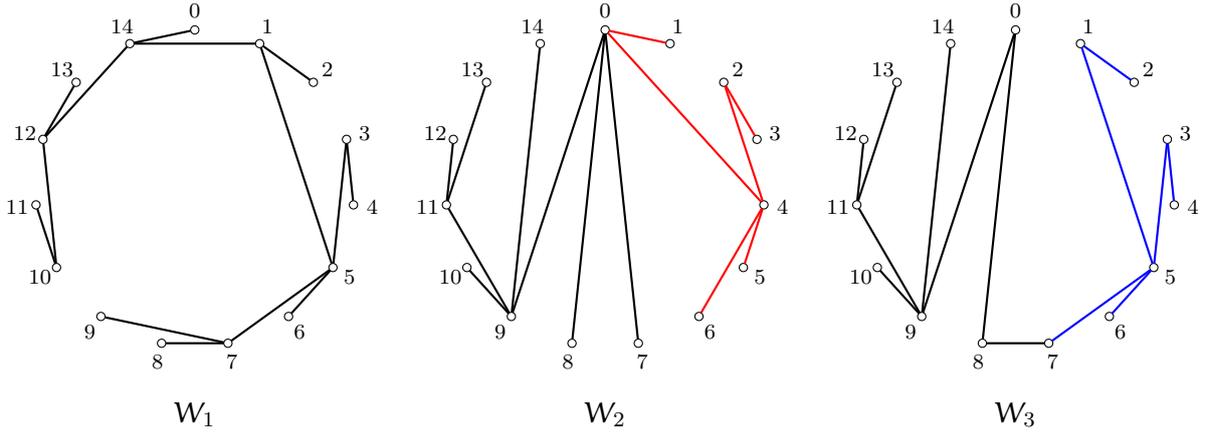

To apply the path method for Markov chain comparison \cite{DS93},
we need to simulate an adjacent move with a sequence of flip moves, 
and to be able to recover the initial and final states of the adjacent move, given the two states of the flip move and at most polynomial amount of additional information.
The main task is to simulate the shift.
We do so by identifying a hierarchical structure of the shift, and only flip edges following a particular pattern. 
This design allows us to uniquely recover the initial and final states of the shift, as long as we know the current transition, the recursion depth, and constant amount of extra information.
This is the crux of our whole argument and is given in \Cref{sec:AM-vs-Flip}.

\Cref{thm:NCST} is shown by combining the comparison argument with the $O(n^3\log n)$ mixing time of the adjacent move chain.
The comparison has an overhead of $O(n^4)$,
which, roughly speaking, consists of $O(n)$ for the encoding, $O(n)$ for the path length, and $O(n^2)$ to account for the transition probability difference.
Finally, together with the $O(n)$ overhead resulted from spectral gaps, we obtain our mixing time bound of $O(n^8\log n)$.

\subsection{Related work}

McShine and Tetali's mixing time argument for convex triangulations \cite{MT97} is similar to our result in spirit,
which is also built on a comparison argument with Wilson's lattice path chain \cite{Wil04}.
However, in their case the starting point is $1$-Dyck path and the argument is more straightforward.
The idiosyncrasies of non-crossing spanning trees warrant a more complicated path construction.

We have mentioned that the down arrows of $2$-Dyck paths form bases of a matroid.
However it appears difficult to directly relate that matroid and related dynamics with the flip chain.
Another very similar structure is the \emph{graphic matroid},
whose bases are spanning trees of a graph.
The bases-exchange walk for matroids is known to mix very rapidly \cite{ALOV19,CGM19}. 
However, with the non-crossing restriction, even when the points are in convex position, NCSTs do \emph{not} form the bases of a matroid if we take all possible edges as the ground set. 
To see this, we give an example where the basis exchange property of matroids fails. 
Consider $4$ points $\{0,1,2,3\}$ on a circle, a tree $T_1$ consisting of edges $(0,1)$, $(1,3)$ and $(2,3)$, and another tree $T_2=\{(0,2), (1,2), (2,3)\}$. 
Remove the edge $(0,1)$ from $T_1$. 
The possible choices of the new edge to add are $T_2\setminus T_1=\{(0,2),(1,2)\}$, 
but neither of them results in a valid non-crossing spanning tree, contradicting to the basis exchange property.
Nonetheless, it is an interesting open problem to utilise techniques for the rapid mixing of graphic spanning trees to obtain better bounds in the non-crossing spanning trees context.

There has been a flurry of new tools building upon \cite{ALOV19,CGM19} for analysing Markov chains,
most notable among which is the spectral independence technique \cite{ALO20}.
However, the results utilising spectral independence or its variants mostly focus on spin systems,
and non-crossing trees are not spin systems.
The trees all have the same size, which is a constraint usually not present in spin systems. 
There are some exceptions, such as the optimal mixing time of the down-up walk of independent sets of the same size \cite{JMPV23}.  
These analyses rely heavily on the relationship between the fixed size distribution with the one without size restrictions.
There is no similar connection to be utilised for non-crossing trees.

\subsection{Open problems}

The most straightforward open problem is to establish the correct order of the mixing time for the NCST flip chain.
The $O(n^8\log n)$ bound in \Cref{thm:NCST} is unlikely to be tight.
Our starting point, the adjacent move chain for $2$-Dyck paths, has diameter $\Omega(n^2)$ by a simple potential argument,\footnote{One can also consider how the area below the path evolves and show that the mixing time is at least $\Omega(n^3/\log n)$, so Wilson's bound is almost tight.}
and thus does not mix very fast.
On the other hand, reconfiguration results \cite{AF96,BKUV24} imply that the flip chain has diameter $O(n)$. 
Therefore, we expect the flip chain to mix at least no slower than the adjacent move chain, instead of the polynomial slow-down in our current bound. 
On the other hand, there is no known lower bound on the mixing time other than the trivial $\Omega(n)$ obtained via the diameter bound. 
It would be interesting to close or shrink this gap. 

Another interesting problem is to study the flip chain when points are \emph{not necessarily} in convex position. 
Recall that the complexity of counting NCSTs with input points in general position, either exactly or approximately, is not known. 
Studying the flip chain would be helpful in resolving the approximate counting complexity.
Constructing a set of points such that the flip chain mixes torpidly (i.e., not mixing in polynomial time) would rule out direct MCMC approaches,
and often torpid-mixing instances help people construct gadgets for hardness reductions.
On the other hand, if this chain always mixes in polynomial time, then efficient approximate counting would be possible.

\section{Markov chain comparison}
\label{sec:markov-chain}

In this section we review some notions of discrete-time Markov chains over discrete state spaces.
For detailed backgrounds we refer the reader to the book \cite{LP17}.
Let $\Omega$ be a finite discrete state space and $\pi$ a distribution on $\Omega$. Let $P:\Omega\times\Omega\rightarrow\mathbb{R}_{\geq0}$ be the transition matrix of a Markov chain with stationary distribution $\pi$. We say $P$ is reversible if $\pi(x)P(x,y)=\pi(y)P(y,x)$ for any $x,y\in\Omega$. 

Given two distributions $\pi$ and $\pi'$ over $\Omega$, the total variation distance between them is defined as
\begin{equation}
d_{\mathrm{TV}}(\pi,\pi')\defeq\frac{1}{2}\sum_{x\in\Omega}\abs{\pi(x)-\pi'(x)}. 
\end{equation}

For a Markov chain $P$ with stationary distribution $\pi$, let
\begin{equation}
d(t)=\max_{x\in\Omega}\left\{d_{\mathrm{TV}}(P^t(x,\cdot),\pi)\right\}. 
\end{equation}

Then the mixing time of $P$ is defined as $t_\mathrm{mix}(P)\defeq\min\{t:d(t)\leq\frac{1}{4}\}$. 
The constant $1/4$ here is not usually important, as the error decays exponentially as $t$ increases.

Let $f,g$ be two real valued functions on $\Omega$. The Dirichlet form of $P$ with the stationary distribution $\pi$ is
\[\mathcal{E}_P(f,g)\defeq\frac{1}{2}\sum_{x,y\in\Omega}[f(x)-f(y)][g(x)-g(y)]\pi(x)P(x,y).\]
Define the variance of a function $f:\Omega\to\mathbb{R}$ by
\[
\Var{\pi}{f}:=\Ex[f^2]-(\Ex[f])^2. 
\]
Then the \emph{spectral gap} of the Markov chain $P$ is
\begin{equation}\label{equ:lambda-def}
\lambda(P):=\inf\left\{\frac{\mathcal{E}_P(f,f)}{\Var{\pi}{f}}\mid f:\Omega\to\mathbb{R},\Var{\pi}{f}\neq0\right\}. 
\end{equation}

The mixing time of the reversible Markov chain $P$ can be bounded using the inverse spectral gap (also called the \emph{relaxation time}) by (see e.g.~\cite[Theorem 12.4]{LP17})
\begin{equation}\label{equ:lambda-to-mixing}
t_{\mathrm{mix}}(P)\leq\frac{1}{\lambda(P)}\left(1+\frac{1}{2}\log\pi_{\min}\right),
\end{equation}
where $\pi_{\min}:=\min_{x\in\Omega}\pi(x)$. 
This $\pi_{\min}$ term is usually a single inverse exponential, which means the mixing time bound obtained this way typically differs from the relaxation time by a linear factor. 

The relaxation time also provides a lower bound for the mixing time. 
For any reversible Markov chain, we have (see e.g.~\cite[Theorem 12.5]{LP17})
\begin{equation}\label{equ:mixing-to-lambda}
\frac{1}{\lambda(P)}\leq 1+2t_{\mathrm{mix}}. 
\end{equation}

%
%
%
%
%
%
%
%
%

\subsection{Path method}

The path method \cite{DS93} is a way to compare two Markov chains $P$ and $\widetilde{P}$, with stationary distributions $\pi$ and $\widetilde\pi$, by simulating any transition in $\widetilde{P}$ using a series of transitions in $P$. 
This is a generalisation of the canonical path method by Jerrum and Sinclair \cite{JS89}.

For every pair of $x,y\in\Omega$, let $\Gamma_{xy}=(x=x_0,x_1,...,x_{\ell-1},x_\ell=y)$ be a path of length $\ell=|\Gamma_{xy}|$ in the state space where $P(x_{i-1},x_i)>0$ for all $i\in[\ell]$. 
The congestion ratio $B$ for a collection of paths $\Gamma=\{\Gamma_{xy}\}_{x,y\in\Omega}$ is defined as
\begin{equation}
B\defeq\max_{(x',y'):P(x',y')>0}\left(\frac{1}{\pi(x')P(x',y')}\sum_{\substack{x,y\\\Gamma_{xy}\ni(x',y')}}\widetilde{\pi}(x)\widetilde{P}(x,y)|\Gamma_{xy}|\right).
\end{equation}

\begin{theorem}[Theorem 2.1 of \cite{DS93}] \label{thm:path-method}
  Let $P$ and $\widetilde{P}$ be two reversible Markov chains. Let $B$ be the congestion ratio for paths $\Gamma=\{\Gamma_{xy}\}$ using transitions in $P$. Then for all $f:\Omega\rightarrow\mathbb{R}$,
\[\mathcal{E}_{\widetilde{P}}(f,f)\leq B\mathcal{E}_P(f,f).\]
\end{theorem}

Roughly speaking, the congestion bounds how much of a slowdown the path method can cause.
It gives an upper bound on how much of a bottleneck can result from our simulated transitions.


\section{Fuss-Catalan numbers and structures}
\label{sec:fuss-catalan}

This section introduces the combinatorial structures that we frequently utilise. 
These structures are closely related to the famous \emph{Fuss-Catalan number}. 

\begin{remark}
Different authors appear to index Fuss-Catalan numbers differently in the literature. 
We adopt the one used in Stanley's monograph \cite{Sta15}. 
\end{remark}

\subsection{Dyck paths} \label{sec:dyck-path}

\begin{definition}[$k$-Dyck path]
Let $k$ be a positive integer. 
A sequence $a_1,a_2,\cdots,a_{(k+1)n}\in\{+1,-k\}$ forms a $k$-\emph{Dyck path} of length $(k+1)n$, if
\[
    \forall j\in[(k+1)n], \qquad
    \sum_{i=1}^{j}a_i\geq 0, \qquad
    \text{and} \qquad 
    \sum_{i=1}^{(k+1)n}a_i=0.
\]
\end{definition}

Imagine that there is a hiking frog starting from the sea level. 
Each time it takes a step forward, it either lifts itself up by $1$ meter, or drops by $k$ meters. 
The two conditions are interpreted as (1) at any point it cannot go below the sea level, and (2) after $(k+1)n$ steps it must land on the sea level. 
This view leads to the following terminology for $k$-Dyck paths. 
We replace $+1$ with an \emph{up-step} $\nearrow$, and $-k$ with a \emph{down-step} $\searrow$. 
The partial sum 
\[
\height(j)\defeq\sum_{i=1}^{j}a_i
\] 
indicates the \emph{height} of the frog after $j$ steps. 
Below is an example of a valid $2$-Dyck path of length $21$. 
\begin{example} \label{emp:dyck-path}
$\nearrow\nearrow\nearrow\nearrow\searrow\nearrow\searrow\nearrow\nearrow\searrow\nearrow\nearrow\nearrow\nearrow\searrow\nearrow\searrow\searrow\nearrow\nearrow\searrow$ is a $2$-Dyck path. 
\end{example}

The following simple fact is useful later. 
\begin{fact} \label{fac:dyck-path-decompose}
Any non-empty $k$-Dyck path $W$ can be uniquely written in the form
\[
  W=\nearrow A_1 \nearrow A_2 \cdots \nearrow A_k \searrow B,
\]
where $A_i$'s and $B$ are (possibly empty) $k$-Dyck paths. 
\end{fact}

The decomposition can be obtained as follows. 
The $B$ part begins at the position where the height \emph{first} reaches $0$ (except the starting point), namely $\mathsf{Start}(B)=\min\{j:~\height_W(j)=0, j\geq 1\}$. 
Each $A_i$ part begins at the position one step after where the height reaches $i-1$ for the \emph{last} time before the $B$ part, 
namely $\mathsf{Start}(A_i)=1+\max\{j:~\height_W(j)=i-1, 0\leq j< \mathsf{Start}(B)\}$.  
See \Cref{fig:dyck-path-decompose} for an example of such a decomposition for a $2$-Dyck path. 

\begin{figure}[t]
\centering
\begin{tikzpicture}[xscale=0.68,yscale=0.68]
\draw[very thin, gray!20, step=1](0,0) grid (21,6);
\draw [thick] [->] (0,0)--(22,0) node[right, below] {step};
    \foreach \x in {0,...,21}
        \draw[thick] (\x,0.1)--(\x,-0.1) node[below] {$\x$};
\draw [thick] [->] (0,0)--(0,7) node[above, right] {height};
    \foreach \x in {0,...,6}
        \draw[thick] (0.1,\x)--(-0.1,\x) node[left] {$\x$};

\edef\lev{0}
\foreach \x in {0,...,20} {
    \ifthenelse{\x = 4 \OR \x = 6 \OR \x = 9 \OR \x = 14 \OR \x = 16 \OR \x = 17 \OR \x = 20}{
            \draw[thick,->] (\x,\lev)--({{\x+1}},{{-2+\lev}});
            \pgfmathparse{-2+\lev}
            \xdef\lev{\pgfmathresult}
        }{
            \draw[thick,->] (\x,\lev)--({{\x+1}},{{1+\lev}});
            \pgfmathparse{1+\lev}
            \xdef\lev{\pgfmathresult}
        }
}
\fill [black!50,fill opacity=0.2] (1,1) rectangle (10,5);
\fill [black!50,fill opacity=0.2] (11,2) rectangle (17,6);
\fill [black!50,fill opacity=0.2] (18,0) rectangle (21,3);

\node at (1.5,4.5) {\Large $A_1$};
\node at (11.5,5.5) {\Large $A_2$};
\node at (18.5,2.5) {\Large $B$};
\end{tikzpicture}
\caption{Decomposing \Cref{emp:dyck-path} by \Cref{fac:dyck-path-decompose}.}
\label{fig:dyck-path-decompose}
\end{figure}

\Cref{fac:dyck-path-decompose} immediately implies the recurrence on the number of $k$-Dyck paths of length $(k+1)n$, denoted by $C_{k,n}$: 
\begin{equation} \label{equ:catalan-fuss}
    C_{k,0}=1,~ C_{k,n}=\sum_{\substack{i_1,\cdots,i_{k+1}\in\mathbb{Z}_{\geq 0},\\\sum_{j=1}^{k+1} i_j=n-1}}C_{k,i_1}C_{k,i_2}\cdots C_{k,i_{k+1}},
    \qquad
    \implies
    \qquad
    C_{k,n}=\frac{1}{kn+1}\binom{(k+1)n}{n}.
\end{equation}
This is called the \emph{$k$-Fuss-Catalan number} \cite[Chapter 4, A14]{Sta15}, a generalisation of the Catalan number.   It is well-known that the number of $1$-Dyck paths equals the Catalan number. 

\subsection{Non-crossing spanning trees}

Non-crossing spanning trees for points in convex position have a one-to-one correspondence with $2$-Dyck paths.
It will be more convenient to consider the following alternative way to draw the trees (see~\cite[Exercise 5.46]{Sta99}).
Instead of putting points in convex position, we arrange them on a line. 
Each edge is represented by a curve, drawn above the line segment. 
A tree is \emph{non-crossing} if and only if it can be drawn this way without any two edges intersecting (except at the endpoints). 
See \Cref{fig:ncst-drawing}(a)(b) for the two drawings of the same tree.
In other words, each edge is represented by a pair of integers $(a,b)$ where $0\leq a<b\leq n$, and there does not exist any two edges $(a,b)$ and $(c,d)$ that $a<c<b<d$. 
As a result, any subtree must consist of points labeled by consecutive numbers. 
For this reason, if a subtree is spanned by $s,s+1,\cdots,t$, we use a shorthand $[s,t]$ for it. 
A subtree is empty (consists of zero edge) if $s=t$. 

\begin{fact} \label{fac:ncst-decompose}
Any $n$-edge non-empty non-crossing spanning tree $T$ on points labeled $0,1,\cdots,n$ can be uniquely decomposed into a tuple of three non-crossing trees $(T_A,T_B,T_C)$, with the edge counts summing up to $n-1$, in the following way:
\begin{itemize}
\item Let $t$ be the maximum index that there is an edge $(0,t)$. 
\item Let $s$ be the index such that, after removing the edge $(0,t)$, the subtree $[0,s]$ is connected, but $[0,s+1]$ is not. In other words, there is a gap between $s$ and $s+1$. 
\item The first tree $T_A$ is the subtree $[0,s]$. It is empty if $s=0$.
\item The second tree $T_B$ is the subtree $[s+1,t]$. It is empty if $t=s+1$. 
\item The last tree $T_C$ is the subtree $[t,n]$. It is empty if $n=t$. 
\end{itemize}
\end{fact}

Resembling \Cref{fac:dyck-path-decompose} for $2$-Dyck paths, the above decomposition not only suggests exactly the same recurrence \eqref{equ:catalan-fuss-2} for the number of $n$-edge non-crossing spanning trees, but also a \emph{one-to-one correspondence} between $2$-Dyck paths and non-crossing spanning trees. 
Compare \Cref{fig:ncst-drawing}(c) and \Cref{fig:dyck-path-decompose} for instance. 
\begin{corollary} \label{cor:path-tree-bij}
The following ruleset, when applied inductively, gives a one-to-one correspondence between $2$-Dyck paths of length $3n$ and non-crossing spanning trees containing $n$ edges.
\begin{itemize}
\item An empty path corresponds to an empty tree. 
\item If a tree $T$ is decomposed into $(T_A,T_B,T_C)$ by \Cref{fac:ncst-decompose}, with each part corresponding to paths $A$, $B$ and $C$ respectively, then $T$ corresponds to the path $\nearrow A \nearrow B \searrow C$. 
\end{itemize}
\end{corollary}

The decomposition has the following useful property. 
\begin{proposition}\label{prop:concatenation}
Let $U$ and $V$ be two $2$-Dyck paths of length $3n_U$ and $3n_V$ respectively, 
and let $T_U$ and $T_V$ be their corresponding non-crossing spanning trees. 
Then, the non-crossing spanning tree $T_{UV}$ corresponding to the concatenated Dyck path $UV$ is a concatenation of $T_U$ and $T_V$, meaning that the subtree $[0,n_U]$ of $T_{UV}$ is $T_U$, and the subtree $[n_U,n_U+n_V]$ of $T_{UV}$ is $T_V$ (up to a shift in indexing). 
\end{proposition}

\begin{proof}
We perform an induction on $k\defeq\left|\left\{i\geq 1:\height_{U}(i)=0\right\}\right|$, the number of times that $U$ hits height zero after the first step. 

In the base case where $k=1$, the decomposition of $U$ is $U=\nearrow U_A \nearrow U_B \searrow U_C$ with $U_C$ being empty. 
Therefore, the decomposition of $UV$ is $UV=\nearrow U_A \nearrow U_B \searrow V$, which by \Cref{cor:path-tree-bij} gives a decomposition of $T_{UV}$ into $(T_{U_A},T_{U_B},T_V)$, a concatenation of $T_U$ and $T_V$. 

Suppose that this is true for any $k'<k$. 
Decompose $U=\nearrow U_A \nearrow U_B \searrow U_C=:U'U_C$. 
On one hand, the number of times that $U_C$ hits height zero is $k-1$. 
By induction hypothesis for the case $k'=k-1$, the tree $T_{U_C V}$ is a concatenation of the tree $T_{U_C}$ and $T_V$. 
On the other hand, the decomposition of $UV$ is $UV=\nearrow U_A \nearrow U_B \searrow (U_C V)=U'(U_C V)$, and note that $U'$ hits height zero once. 
By induction hypothesis for the case $k'=1$, we have that $T_{UV}=T_{U'(U_C V)}$ is a concatenation of $T_{U'}$ and $T_{U_C V}$. 
The induction hypothesis is hence true for the case $k'=k$ by putting the above two points together. 
\end{proof}

\begin{figure}[t]
\centering
\begin{minipage}{0.45\textwidth}
\centering
\begin{tikzpicture}
\foreach \x in {0,...,7} {
    \node [shape=circle,draw=black,fill=white,scale=0.75] (A\x) at ({90-360/8*\x}:3) {};
    \node () at ({90-360/8*\x}:3.4) {$\x$};
}
\path [thick] (A0) edge node [] {} (A1);
\path [thick] (A0) edge node [] {} (A2);
\path [thick] (A2) edge node [] {} (A3);
\path [thick] (A4) edge node [] {} (A5);
\path [thick] (A4) edge node [] {} (A6);
\path [thick] (A6) edge node [] {} (A7);
\path [thick] (A0) edge node [] {} (A6);
\end{tikzpicture}
\caption*{(a)}
\end{minipage}
\begin{minipage}{0.53\textwidth}
\centering
\begin{minipage}{\linewidth}
\centering
\begin{tikzpicture}
\foreach \x in {0,...,7} {
    \node [shape=circle,draw=black,fill=white,scale=0.75] (A\x) at ({\x},0) {};
    \node () at ({\x},-0.4) {$\x$};
}
\path [thick,bend left=30] (A0) edge node [] {} (A1);
\path [thick,bend left=45] (A0) edge node [] {} (A2);
\path [thick,bend left=30] (A2) edge node [] {} (A3);
\path [thick,bend left=30] (A4) edge node [] {} (A5);
\path [thick,bend left=45] (A4) edge node [] {} (A6);
\path [thick,bend left=30] (A6) edge node [] {} (A7);
\path [thick,bend left=60] (A0) edge node [] {} (A6);
\end{tikzpicture}
\caption*{(b)}
\end{minipage}

\begin{minipage}{\linewidth}
\centering
\begin{tikzpicture}
\foreach \x in {0,...,7} {
    \node [shape=circle,draw=black,fill=white,scale=0.75] (A\x) at ({\x},0) {};
    \node () at ({\x},-0.4) {$\x$};
}
\path [thick,bend left=30] (A0) edge node [] {} (A1);
\path [thick,bend left=45] (A0) edge node [] {} (A2);
\path [thick,bend left=30] (A2) edge node [] {} (A3);
\path [thick,bend left=30] (A4) edge node [] {} (A5);
\path [thick,bend left=45] (A4) edge node [] {} (A6);
\path [thick,bend left=30] (A6) edge node [] {} (A7);
\path [thick,bend left=60] (A0) edge node [] {} (A6);

\fill [black,fill opacity=0.2] (1.5,0) ellipse (1.8 and 1);
\fill [black,fill opacity=0.2] (5,0) ellipse (1.3 and 0.8);
\fill [black,fill opacity=0.2] (6.5,0) ellipse (0.8 and 0.5);

\node at (1.5,-1.3) {\Large $T_A$};
\node at (5.0,-1.1) {\Large $T_B$};
\node at (6.5,-0.8) {\Large $T_C$};
\end{tikzpicture}
\caption*{(c)}
\end{minipage}
\end{minipage}
\caption{(a)(b) Two ways of representing a non-crossing spanning tree. 
The tree is corresponded to the $2$-Dyck path in \Cref{emp:dyck-path}. 
(c) Decomposing the tree by \Cref{fac:ncst-decompose}.}
\label{fig:ncst-drawing}
\end{figure}

Inspired by the decomposition, we introduce the following terminology that will be useful later. 

\begin{definition}
In the context of \Cref{fac:ncst-decompose}: 

{
\renewcommand{\arraystretch}{1.5}
\begin{tabular}{R{0.20\linewidth} L{0.74\linewidth}}
{(pivot edge)} & the edge $(0,t)$ is called the \emph{pivot edge} of $T$; \\
{(gap)} & the gap after $s$ is called the \emph{gap} beneath the pivot edge $(0,t)$; 
the pivot edge is called the \emph{overarching edge} of the gap; \\
{(overarching edge)} & an edge $(a,b)$ is the \emph{overarching edge} of another edge $(c,d)$, if $a\leq c<d\leq b$, and there does not exist any other edge $(a',b')$ that $a\leq a'\leq c<d\leq b'\leq b$;\\
{(minimal segment)} & a subtree $[c,d]$ is a \emph{minimal segment} under an edge $(a,b)$, if the edge $(c,d)$ exists, and $(a,b)$ is the overarching of $(c,d)$; 
a subtree $[c,d]$ is an outmost minimal segment, if the edge $(c,d)$ exists, and there is no overarching edge of $(c,d)$.\\
\end{tabular}
}

\end{definition}

\begin{example}
Consider the tree in \Cref{fig:ncst-drawing}.
\begin{itemize}
    \item The pivot edge of the whole tree is $(0,6)$. The pivot edge of the subtree $[0,3]$ is $(0,2)$. 
    \item The gap beneath the edge $(0,6)$ is after $3$. The gap beneath the edge $(4,6)$ is after $5$. 
    \item The overarching edge of the edge $(0,1)$ is $(0,2)$, whose overarching edge is $(0,6)$, which does not have any overarching edge. 
    \item The minimal segments under the edge $(0,6)$ are 
    $[0,2]$, $[2,3]$ and $[4,6]$. 
    The outmost minimal segments are
    $[0,6]$ and $[6,7]$.
\end{itemize} 
\end{example}



\section{Spectral gap of the adjacent move chain}

Over the state space of $2$-Dyck paths of length $3n$, define the following \emph{adjacent move chain}: 
\begin{enumerate}
  \item[(1)] starting from the current path $X$, choose a position $i\in[3n-1]$;
  \item[(2)] let $X'$ be the same as $X$ except that the steps at positions $i$ and $i+1$ are swapped;
  \item[(3)] if $X'$ is invalid, then the chain stays at $X$. Otherwise, it moves to $X'$ with probability $1/2$ and stays at $X$ with probability $1/2$. 
\end{enumerate}
Formally, let $B(X)$ be the set of $2$-Dyck paths that can be obtained from $X$ by swapping two adjacent coordinates. 
Trivially, $|B(X)|<n$. 
The transition probability of this chain is given by
\begin{equation} \label{equ:am-chain}
P_{\mathrm{AM}}(X,X')=
\begin{cases}
\displaystyle
1-\frac{B(X)}{6n-2} & \qquad \text{if $X=X'$}; \\
\displaystyle
\frac{1}{6n-2} & \qquad \text{if $X'\in B(X)$}; \\
\displaystyle
0 & \qquad \text{otherwise}.
\end{cases}
\end{equation}
Note that $P_{\mathrm{AM}}(X,X)\geq 4/5$. 
It is easy to see that $P_\mathrm{AM}$ is irreducible and verify that $P(X,X')=P(X',X)$, so $P_\mathrm{AM}$ is reversible. 

We have the following bound on the mixing time of this chain. 
\begin{theorem} \label{thm:wilson-mixing}
$t_{\mathrm{mix}}(P_{\mathrm{AM}})=O(n^3\log n)$. 
\end{theorem}

The proof essentially follows Wilson's coupling argument \cite{Wil04}. 
The core is to define a potential function measuring the difference between two coupled Markov chains. Recall that
\[
\height(j)\defeq\sum_{i=1}^{j}a_i
\] 
where each $a_i$ is either $1$ (for an up-step) or $-2$ for a down-step.
We say that a path $X$ \emph{dominates} another path $Y$ if $\height_{X}(i)\geq\height_{Y}(i)$ for any $i$. 
Wilson defines a potential $\Phi(X,Y)$ (only when $X$ dominates $Y$) which is a weighted sum of height differences. When we apply it to 2-Dyck paths (a subset of lattice paths on a lattice with width $n$ and height $2n$), it maintains the following properties: 
\begin{enumerate}
\item[(a)] $\Phi(X,Y)\geq 0$, and the equality is taken if and only if $X=Y$;
\item[(b)] $\sin(\pi/(3n))\leq\Phi(X,Y)\leq 4n^2/27$ for any pair of two different $2$-Dyck paths of length $3n$. 
\end{enumerate}

Consider two instances of Markov chains $X_t$ and $Y_t$ where the former dominates the latter. 
In the next step, the coupling between $X_{t+1}$ and $Y_{t+1}$ is defined as follows:
\begin{itemize}
\item choose the same positions in both chains for the potential swap;
\item if, in either $X_t$ or $Y_t$ (or both), the two positions to swap have the same arrow, then perform the transition in the other chain (or neither) as normal;
\item otherwise, choose either $\nearrow\searrow$ or $\searrow\nearrow$ with probability $1/2$, respectively, and change the two positions into the pattern chosen in both chains. For any chain resulting in an invalid Dyck path, revert the swap in the invalid chain. 
\end{itemize}
Let $\Phi_t$ be the shorthand for $\Phi(X_{t},Y_{t})$. Wilson also showed that
\begin{enumerate}
\item[(c)] Given any pair of $X_t,Y_t$ where $X_t$ dominates $Y_t$, obtaining $X_{t+1}$ and $Y_{t+1}$ in the above way. Then $X_{t+1}$ still dominates $Y_{t+1}$, and it holds that
\[
\Ex[\Phi_{t+1}-\Phi_{t}\mid X_t,Y_t]\leq -\frac{\pi^2}{2n^3} \Phi_{t}.
\]
\end{enumerate}

The coupling here differs from the one in Wilson's in the extra check of whether the swap results in a valid path.
This does break some statements in Wilson's paper, such as the mixing time lower bound, but it is easy to verify that the properties a), b), and c) still hold.

\begin{proof}[Proof of \Cref{thm:wilson-mixing}]
  Consider two instances of the adjacent move chain, whose initial states are $X_0$ and $Y_0$.
  Here $X_0$ consists of $2n$ $\nearrow$s followed by $n$ $\searrow$s, and $Y_0$ consists of $\nearrow\nearrow\searrow$ repeated $n$ times. 
  It is easy to verify that any valid $2$-Dyck path is dominated by $X_0$ and dominates $Y_0$. 
  By the above properties (b) and (c), we have
  \[
    \sin\left(\frac{\pi}{3n}\right)\Pr[\Phi_{t}>0]\leq\Ex[\Phi_{t}]\leq \frac{4n^2}{27}\exp\left\{-\frac{\pi^2t}{2n^3}\right\}.
  \]
  Therefore, after $t>(2/\pi^2)n^3\log(16n^2/27/\sin(\pi/(3n)))=\Theta(n^3\log n)$ steps, the probability that $\Phi_t$ is non-zero,
  and hence that $X_t$ and $Y_t$ do not coalesce yet, is at most $1/4$. 
\end{proof}

\begin{corollary} \label{cor:wilson-spectral-gap}
  $1/\lambda(P_{\mathrm{AM}})=O(n^3\log n)$.
\end{corollary}

\begin{proof}
  This follows immediately from \Cref{thm:wilson-mixing} and \eqref{equ:mixing-to-lambda}.
\end{proof}

\section{Adjacent move versus flip move} \label{sec:AM-vs-Flip}

\tikzset{
ellipsebubble/.pic = {
    \draw[very thick] (0.5,0) ellipse (0.5 and 0.25);
}
}

\tikzset{
shift/.pic = {
\pic at (1,0) {ellipsebubble};
\pic at (4,0) {ellipsebubble};

\draw[very thick] (0,0) .. controls (0.5,0.75) and (1.5,0.75) .. (2,0);
\draw[very thick] (4,0) .. controls (4.5,0.75) and (5.5,0.75) .. (6,0);

\path [->,very thick] (2.5,0) edge (3.5,0);
}
}

%
%

\tikzset{
flip2/.pic = {
\pic at (0,0) {ellipsebubble};
\pic at (1,0) {ellipsebubble};
\pic at (3,0) {ellipsebubble};
\pic at (5,0) {ellipsebubble};

\draw[very thick] (0,0) arc (180:0:3 and 1.5);
\draw[very thick] (1,0) arc (180:0:1 and 0.75);

\node at (0.5,0) {$A_1$};
\node at (1.5,0) {$A_2$};
\node at (3.5,0) {$B_1$};
\node at (5.5,0) {$B_2$};
\node at (3,1.7) {$a$};
\node at (2,1) {$e$};
}
}

\tikzset{
flip3/.pic = {
\pic at (0,0) {ellipsebubble};
\pic at (2,0) {ellipsebubble};
\pic at (3,0) {ellipsebubble};
\pic at (5,0) {ellipsebubble};

\draw[very thick] (0,0) arc (180:0:3 and 1.5);
\draw[very thick] (1,0) arc (180:0:1 and 0.75);

\node at (0.5,0) {$A_1$};
\node at (2.5,0) {$A_2$};
\node at (3.5,0) {$B_1$};
\node at (5.5,0) {$B_2$};
}
}

\tikzset{
flip4/.pic = {
\pic at (0,0) {ellipsebubble};
\pic at (2,0) {ellipsebubble};
\pic at (3,0) {ellipsebubble};
\pic at (5,0) {ellipsebubble};

\draw[very thick] (0,0) arc (180:0:3 and 1.5);
\draw[very thick] (0,0) arc (180:0:1.5 and 0.75);

\node at (0.5,0) {$A_1$};
\node at (2.5,0) {$A_2$};
\node at (3.5,0) {$B_1$};
\node at (5.5,0) {$B_2$};
}
}

\tikzset{
flip5/.pic = {
\pic at (0,0) {ellipsebubble};
\pic at (2,0) {ellipsebubble};
\pic at (3,0) {ellipsebubble};
\pic at (5,0) {ellipsebubble};

\draw[opacity=0] (0,0) arc (180:0:3 and 1.5);
\draw[very thick] (3,0) arc (180:0:1.5 and 0.75);
\draw[very thick] (0,0) arc (180:0:1.5 and 0.75);

\node at (0.5,0) {$A_1$};
\node at (2.5,0) {$A_2$};
\node at (3.5,0) {$B_1$};
\node at (5.5,0) {$B_2$};
}
}

\tikzset{
flip6/.pic = {
\pic at (0,0) {ellipsebubble};
\pic at (2,0) {ellipsebubble};
\pic at (3,0) {ellipsebubble};
\pic at (5,0) {ellipsebubble};

\draw[very thick] (0,0) arc (180:0:3 and 1.5);
\draw[very thick] (3,0) arc (180:0:1 and 0.75);

\node at (0.5,0) {$A_1$};
\node at (2.5,0) {$A_2$};
\node at (3.5,0) {$A_3$};
\node at (5.5,0) {$B$};
\node at (3,1.7) {$a$};
\node at (4,1) {$e$};
}
}

\tikzset{
flip7/.pic = {
\pic at (0,0) {ellipsebubble};
\pic at (2,0) {ellipsebubble};
\pic at (3,0) {ellipsebubble};
\pic at (5,0) {ellipsebubble};

\draw[very thick] (0,0) arc (180:0:3 and 1.5);
\draw[very thick] (1,0) arc (180:0:1 and 0.75);

\node at (0.5,0) {$A_1$};
\node at (2.5,0) {$A_2$};
\node at (3.5,0) {$A_3$};
\node at (5.5,0) {$B$};
}
}

\tikzset{
flip8/.pic = {
\pic at (0,0) {ellipsebubble};
\pic at (1,0) {ellipsebubble};
\pic at (3,0) {ellipsebubble};
\pic at (5,0) {ellipsebubble};

\draw[very thick] (0,0) arc (180:0:3 and 1.5);
\draw[very thick] (1,0) arc (180:0:1 and 0.75);

\node at (0.5,0) {$A_1$};
\node at (1.5,0) {$A_2$};
\node at (3.5,0) {$A_3$};
\node at (5.5,0) {$B$};
}
}

\tikzset{
flip9/.pic = {
\pic at (0,0) {ellipsebubble};
\pic at (1,0) {ellipsebubble};
\pic at (3,0) {ellipsebubble};
\pic at (5,0) {ellipsebubble};

\draw[very thick] (0,0) arc (180:0:3 and 1.5);
\draw[very thick] (1,0) arc (180:0:1.5 and 0.75);

\node at (0.5,0) {$A_1$};
\node at (1.5,0) {$A_2$};
\node at (3.5,0) {$A_3$};
\node at (5.5,0) {$B$};
}
}

In this section we analyse the mixing time of the flip chain by comparing it with the adjacent move chain and using \Cref{cor:wilson-spectral-gap}.
The flip chain $P_{\mathrm{FM}}$ is a random walk on non-crossing spanning trees (or non-crossing trees),
where in each step we remove an edge uniformly at random, and add back a non-crossing edge between the two components.
Such a move is called an (edge) flip.
Formally, given two non-crossing spanning trees $S$ and $T$, the probability of moving from $S$ to $T$ is
\begin{equation} \label{equ:fm-chain}
P_{\mathrm{FM}}(S,T)=
\begin{cases}
\displaystyle
\sum_{T':~|S\cap T'|=n-1}\frac{1}{n\delta(S,T')} & \qquad \text{if $S=T$}; \\
\displaystyle
\frac{1}{n\delta(S,T)} & \qquad \text{if $|S\cap T|=n-1$}; \\
\displaystyle
0 & \qquad \text{otherwise},
\end{cases}
\end{equation}
where $\delta(S,T)$ is the number of edges that one can add to $S\cap T$ to make it a non-crossing spanning tree. 
Note that 
\begin{align}  \label{eqn:fm-transition-bound}
  P_{\mathrm{FM}}(S,S)\geq 1/n,\quad\text{ and }P_{\mathrm{FM}}(S,T)\geq \Omega(n^{-3})\text{ if it is non-zero.}
\end{align}

It is easy to see that $P_\mathrm{FM}$ is irreducible. Note when $|X\cap Y|=n-1$, $P(X,Y)=\frac{1}{n\delta(S,T)}=\frac{1}{n\delta(T,S)}=P(Y,X)$, and it is trivial to see that $P(X,Y)=P(Y,X)$ in the other two cases, so $P_\mathrm{FM}$ is symmetric, thus is reversible. 

\subsection{Characterise the adjacent move via flip moves}
\label{sec:characterise-adjacent}

A single adjacent move on non-crossing spanning trees still requires altering an $\Omega(n)$ number of edges. 
Yet, we will show that it consists of a constant number of flip moves, plus one subtree \emph{shifting}, taking the form of 
\begin{equation} \label{equ:shifting-illustration}
\text{moving from}\quad
\begin{tikzpicture}[scale=0.75,baseline=-.5*(height("$+$")-depth("$+$"))]
\draw[very thick] (0,-0.15) arc (180:0:1 and 0.6);
\draw[very thick] (0.75,-0.15) ellipse (0.75 and 0.25);
\end{tikzpicture}
\quad\text{to}\quad
\begin{tikzpicture}[scale=0.75,baseline=-.5*(height("$+$")-depth("$+$"))]
\draw[very thick] (-0.5,-0.15) arc (180:0:1 and 0.6);
\draw[very thick] (0.75,-0.15) ellipse (0.75 and 0.25);
\end{tikzpicture}
\quad\text{or vice versa,}
\end{equation}
where the bubble does not change.
Note that the overarching edge of the subtree remains unchanged. 
This is formally defined as follows, using the non-crossing tree decomposition. 
\begin{definition}[shifting] \label{def:shifting}
Let $T$ be a non-crossing spanning tree of size $m$ decomposing into $(T_A,T_B,T_C)$ by \Cref{fac:ncst-decompose}, where $T_B$ and $T_C$ are empty.
Let $T'$ be another non-crossing spanning tree of size $m$ decomposing into $(T'_A,T'_B,T'_C)$, where both $T'_A$ and $T'_C$ are empty, and $T'_B=T_A$. 
A \emph{shifting} is the transformation from $T$ to $T'$ (right shifting) or from $T'$ to $T$ (left shifting). 
\end{definition}

This particular form of shifting may look strange at first sight.
A more intuitive version would be just moving a subtree left or right, without caring about the overarching edge or the other side of the gap.
However, as it will become clear later when we bound the congestion of the paths,
it is this form of shifting that allows an efficient encoding.

\begin{figure}[p]
\centering
\begin{minipage}{0.49\linewidth}
\centering
\begin{tikzpicture}[xscale=0.73,yscale=0.73]
\draw[very thin, gray!20, step=1](0,0) grid (10,4);
\draw [thick] [-] (0,0)--(10,0);
    \foreach \x in {0,...,10}
        \draw[thick] (\x,0.1)--(\x,-0.1);

\draw[thick,->] (0,0)--(1,1);
\draw[thick,->] (2,1)--(3,2);
\draw[thick,blue,->] (4,2)--(5,3);
\draw[thick,blue,->] (5,3)--(6,1);
\draw[thick,->] (7,1)--(8,2);
\draw[thick,->] (9,2)--(10,0);
\draw[thick,red,->] (4,2)--(5,0);
\draw[thick,red,->] (5,0)--(6,1);

\fill [black!50,fill opacity=0.2] (1,1) rectangle (2,2);
\fill [black!50,fill opacity=0.2] (3,2) rectangle (4,3);
\fill [black!50,fill opacity=0.2] (6,1) rectangle (7,2);
\fill [black!50,fill opacity=0.2] (8,2) rectangle (9,3);

\node at (1.5,1.5) {\Large $A_1$};
\node at (3.5,2.5) {\Large $A_2$};
\node at (6.5,1.5) {\Large $B_1$};
\node at (8.5,2.5) {\Large $B_2$};

\node at (5,-0.4) {$x$};
\node at (6,-0.4) {$y$};
\node at (0,-0.4) {$x'$};
\node at (10,-0.4) {$y'$};
\end{tikzpicture}
\caption*{(1a)}
\end{minipage}
\begin{minipage}{0.49\linewidth}
\centering
\begin{tikzpicture}[xscale=0.73,yscale=0.73]
\draw[very thin, gray!20, step=1](11,0) grid (21,4);
\draw [thick] [-] (11,0)--(21,0);
    \foreach \x in {11,...,21}
        \draw[thick] (\x,0.1)--(\x,-0.1);

\draw[thick,->] (11,0)--(12,1);
\draw[thick,->] (13,1)--(14,2);
\draw[thick,->] (15,2)--(16,3);
\draw[thick,blue,->] (17,3)--(18,4);
\draw[thick,blue,->] (18,4)--(19,2);
\draw[thick,->] (20,2)--(21,0);
\draw[thick,red,->] (17,3)--(18,1);
\draw[thick,red,->] (18,1)--(19,2);

\fill [black!50,fill opacity=0.2] (12,1) rectangle (13,2);
\fill [black!50,fill opacity=0.2] (14,2) rectangle (15,3);
\fill [black!50,fill opacity=0.2] (16,3) rectangle (17,4);
\fill [black!50,fill opacity=0.2] (19,2) rectangle (20,3);

\node at (12.5,1.5) {\Large $A_1$};
\node at (14.5,2.5) {\Large $A_2$};
\node at (16.5,3.5) {\Large $A_3$};
\node at (19.5,2.5) {\Large $B$};

\node at (18,-0.4) {$x$};
\node at (19,-0.4) {$y$};
\node at (11,-0.4) {$x'$};
\node at (21,-0.4) {$y'$};
\end{tikzpicture}
\caption*{(2a)}
\end{minipage}

\vspace{20pt}

\begin{minipage}{0.49\linewidth}
\centering
\begin{tikzpicture}
\pic at (0,12) {flip2};
\pic at (0,8) {flip3};
\pic at (0,4) {flip4};
\pic at (0,0) {flip5};

\path [->,very thick] (3,11) edge node [midway, right] {\large Shift sequence} (3,10);
\path [->,very thick] (3,7) edge node [midway, right] {\large $M_1$} (3,6);
\path [->,very thick] (3,3) edge node [midway, right] {\large $M_2$} (3,2);
\end{tikzpicture}
\caption*{(1b)}
\end{minipage}
\begin{minipage}{0.49\linewidth}
\centering
\begin{tikzpicture}
\pic at (0,12) {flip6};
\pic at (0,8) {flip7};
\pic at (0,4) {flip8};
\pic at (0,0) {flip9};

\path [->,very thick] (3,11) edge node [midway, right] {\large $M_3$} (3,10);
\path [->,very thick] (3,7) edge node [midway, right] {\large Shift sequence} (3,6);
\path [->,very thick] (3,3) edge node [midway, right] {\large $M_4$} (3,2);
\end{tikzpicture}
\caption*{(2b)}
\end{minipage}
\caption{($\ast$a) The $2$-Dyck path representations of $\tilde{I}$. In the adjacent move, the blue part is changed into red. 
($\ast$b) The corresponding flip move sequences. 
(1$\ast$) Corresponding to \textbf{Type 1}. 
(2$\ast$) Corresponding to \textbf{Type 2}.}
\label{fig:FM-sequences}
\end{figure}
\afterpage{\clearpage}

Next is the characterisation of an adjacent move via flip moves and shifting. 
\begin{lemma} \label{lem:AM-simulation}
Let $I\to F$ be a transition in the adjacent move $P_{\mathrm{AM}}$ such that $I\neq F$. 
This transition can be simulated from $T_I$ to $T_F$ by at most two edge flippings in $P_{\mathrm{FM}}$, plus at most one shifting. 
\end{lemma}

There are two cases in this lemma, depending on whether one or two steps the first time the $2$-Dyck path goes below the end of the adjacent move,
which we call Type-1 and Type-2 moves.
See \Cref{fig:FM-sequences} (1a) and (2a) for illustrations of the two types of transitions.
The flip moves to simulate the adjacent move are also intuitive. 
See the illustrations in \Cref{fig:FM-sequences} (1b) and (2b).

\begin{proof}
Without loss of generality, assume $F$ is reached from $I$ with an adjacent move to the left, so $F=(I\setminus\{y\})\cup\{x\}$ where $y=x+1$. 
Let $\tilde{I}$ be the sub-Dyck-path of $I$ from $x'$ to $y'$, obtained in the following way. 
\begin{itemize}
\item $y'>y$ is the minimum $y'$ such that $\height_I(y')=\height_I(y)-1$ or $\height_I(y')=\height_I(y)-2$. 
\item $x'<x$ is the maximum $x'$ such that $\height_I(x')=\height_I(y')$. 
\end{itemize}
Note that such $y'$ must exist, because the down-step only changes the height by $-2$. 
The point here is that we want to find $x'$ and $y'$ so that our transition can be run only on the sub-Dyck path (or subtree) made by the interval $(x',y')$ without changing anything on the rest of the structure. 

We call the transition $I\to F$ \textbf{Type 1} if $\height_I(y')=\height_I(y)-1$, and \textbf{Type 2} if $\height_I(y')=\height_I(y)-2$. 
See \Cref{fig:FM-sequences} (1a) and (2a) for examples. 
We will decompose our sub-Dyck path again into blocks of unaffected sub-sub-Dyck paths ($A_i$'s, $B_i$'s) which do not change in their interior, and remaining arrows which matter in our transition. 

For a \textbf{Type 1} transition,
the sub-Dyck-path $\tilde{I}$ can be written as $\nearrow A_1 \nearrow A_2 \nearrowblue \searrowblue B_1 \nearrow B_2 \searrow$, 
which becomes $\tilde{F}=\nearrow A_1 \nearrow A_2 \searrowred \nearrowred B_1 \nearrow B_2 \searrow$ after the adjacent move. 
See \Cref{fig:FM-sequences} (1a). 
In $T_{\tilde{I}}$, we find the subtrees $T_{A_1},T_{A_2},T_{B_1}$ and $T_{B_2}$ that correspond to $A_1,A_2,B_1$ and $B_2$ respectively. 
To see this, first notice that $A_1 \nearrow A_2 \nearrowblue \searrowblue B_1$ is the ``$A_1$ part'' in \Cref{fac:dyck-path-decompose}, and $B_2$ is the ``$A_2$ part'' in \Cref{fac:dyck-path-decompose}.
The ``$B$ part'' is after $y'$ and is not relevant here, so we do not draw it out.
By \Cref{prop:concatenation}, we further decompose the former into $A_1$ and $\nearrow A_2 \nearrowblue \searrowblue B_1$.
Notice that $A_1$ may not be a minimal segment here.
Then we use \Cref{fac:dyck-path-decompose} on $\nearrow A_2 \nearrowblue \searrowblue B_1$ to further decompose it.

There are two distinguished edges $a$ and $e$ in the corresponding tree as labelled in \Cref{fig:FM-sequences} (1a). 
The edge $a=(a_1,a_2)$ is the pivot edge of $T_{\tilde{I}}$. 
Denote by $g_a,g_a+1$ the corresponding gap beneath $a$. 
The edge $e=(p,g_e+1)$ has the overarching edge $a$. 
Denote by $g_e,g_e+1$ the corresponding gap beneath $e$. 
We can alternatively say that $e$ is the edge between $A_2$ and $B_1$, and $a$ is its overarching edge. 
Therefore, the subtree $T_{A_1}$ is $[a_1,p]$, 
the subtree $T_{A_2}$ is $[p,g_e]$, 
the subtree $T_{B_1}$ is $[g_e+1,g_a]$, 
and the subtree $T_{B_2}$ is $[g_a+1,a_2]$. 
In $T_{\tilde{F}}$, the three subtrees $T_{A_1},T_{B_1},T_{B_2}$ remain untouched, while $T_{A_2}$ starts at position $p+1$ instead of $p$. 
The two edges $a$ and $e$ are moved to new positions, too. 

The above change can be simulated as below by using edge flippings and a shifting. 
See \Cref{fig:FM-sequences}(1b).
\begin{itemize}
\item[($S_1$)] The position of $e$ together with $T_{A_2}$ meets the prerequisite for a right shifting, we perform it so long as $T_{A_2}$ contains at least one edge. 
\item[($M_1$)] Now the subtree $T_{A_2}$ has been shifted right. 
Let $M_1$ move $e$ from $(p,g_e+1)$ to $(a_1,g_e+1)$. 
Since the right endpoint is still connected to the subtree induced by $[p+1,g_a]$, the result is still a tree, and since there is no edge other than $a$ and $e$ going out of $T_{A_1}$, the result is still non-crossing. 
\item[($M_2$)] Let $M_2$ move $a$ from $(a_1,a_2)$ to $(g_e+1,a_2)$. 
Since the left endpoint is still connected to the subtree induced by $[a_1,g_a]$, the result is still a tree, and since there are no edges coming out from beneath the current edge $e$, the result is still non-crossing. 
\end{itemize}

For a \textbf{Type 2} transition, 
the sub-Dyck-path $\tilde{I}$ can be written as $\nearrow A_1 \nearrow A_2 \nearrow A_3 \nearrowblue \searrowblue B \searrow$, 
which becomes $\tilde{F}=\nearrow A_1 \nearrow A_2 \nearrow A_3 \searrowred \nearrowred B \searrow$ after the adjacent move. 
See \Cref{fig:FM-sequences}(2a). 
Again, in $T_{\tilde{I}}$, we find the corresponding subtrees $T_{A_1},T_{A_2},T_{A_3}$ and $T_{B}$. 
They are joined by two edges $a$ and $e$. 
The edge $a=(a_1,a_2)$ is the pivot edge of $T_{\tilde{I}}$. 
Denote by $g_a,g_a+1$ the corresponding gap beneath $a$. 
The edge $e=(q,g_e+1)$ has the overarching edge $a$. 
Denote by $g_e,g_e+1$ the corresponding gap beneath $e$. 
Therefore, the subtree $T_{A_1}$ is $[a_1,g_a]$, 
the subtree $T_{A_2}$ is $[g_a+1,q]$, 
the subtree $T_{A_3}$ is $[q,g_e]$, 
and the subtree $T_{B}$ is $[g_e+1,a_2]$. 
In $T_{\tilde{F}}$, the three subtrees $T_{A_1},T_{A_3},T_{B}$ remain untouched, while $T_{A_2}$ starts at position $g_a$ instead of $g_a+1$. 
The edge $e$ is moved to a new position. 

The above change can be simulated as below by using edge flippings and a shifting. 
See \Cref{fig:FM-sequences}(2b).
\begin{itemize}
\item[($M_3$)] Let $M_3$ move $e$ from $(q,g_e+1)$ to $(g_a,q)$. 
Since $e$ is still has exatly one endpoint connected to the subtree induced by $[g_a+1,g_e]$, the result is still a tree, and since there is no edge other than $e$ coming out from $A_2$, the result is still non-crossing. 
\item[($S_2$)] The position of $e$ together with $T_{A_2}$ now meets the prerequisite for a left shifting, we perform it so long as $T_{A_2}$ contains at least one edge. 
\item[($M_4$)] Let $M_4$ move $e$ from $(g_a,q)$ to $(g_a,g_e)$. Since there is no edge other than $e$ coming out from the subtree $T_{A_3}$, the result is still non-crossing. 
\end{itemize}

These validate the lemma.
\end{proof}

The next step is to resolve the shifting, by designing a sequence of flippings that simulates it. 
We call this sequence of flip moves a \emph{shift sequence}. 
The design of the shift sequence is the trickiest part of the whole comparison argument, because we have to be able to recover the initial and final states by just looking at a transition with some succinct extra encoding. 
We only need to deal with right shifting, as for left shifting we just reverse the sequence. 

\begin{figure}[t]
\centering
\begin{minipage}{0.07\textwidth}
\begin{tikzpicture}
\path [->,draw=none] (0,0) (0.7,0);
\path [->,draw=none] (0,-1.4) (0.7,-1.4);
\end{tikzpicture}
\end{minipage}
\hspace{-0.03\textwidth}
\begin{minipage}{0.45\textwidth}
\begin{tikzpicture}[xscale=0.85,yscale=0.85]
\draw[very thick] (0.6,0) ellipse (0.6 and 0.3);
\draw[very thick] (2.4,0) ellipse (0.6 and 0.3);
\draw[very thick,red] (3.4,0) ellipse (0.4 and 0.3);
\draw[very thick,red] (4.6,0) ellipse (0.4 and 0.3);
\draw[very thick] (6.0,0) ellipse (0.6 and 0.3);
\draw[very thick] (7.8,0) ellipse (0.6 and 0.3);
\draw[very thick] (0,0) arc (180:0:4.2 and 2);
\draw[very thick] (3,0) arc (180:0:1 and 0.8);
\node at (1.5,0) {...};
\node at (3.4,0) {$S_A$};
\node at (4.6,0) {$S_B$};
\node at (4.0,-0.6) {$S_i$};
\node at (6.9,0) {...};
\node at (4.2,2.3) {$a$};
\node at (4.0,1.1) {$e$};
\node at (0.6,-0.6) {$S_1$};
\node at (2.4,-0.6) {$S_{i-1}$};
\node at (6.0,-0.6) {$S_{i+1}$};
\node at (7.8,-0.6) {$S_{k}$};
\end{tikzpicture}
\end{minipage}
\hspace{-0.03\textwidth}
\begin{minipage}{0.07\textwidth}
\begin{tikzpicture}
\path [->,very thick] (0,0) edge node [midway, above] {\large $M_5$} (1.0,0);
\path [->,draw=none] (0,-1.8) (1.0,-1.8);
\end{tikzpicture}
\end{minipage}
\hspace{-0.03\textwidth}
\begin{minipage}{0.45\textwidth}
\begin{tikzpicture}[xscale=0.85,yscale=0.85]
\draw[very thick] (0.6,0) ellipse (0.6 and 0.3);
\draw[very thick] (2.4,0) ellipse (0.6 and 0.3);
\draw[very thick,red] (3.4,0) ellipse (0.4 and 0.3);
\draw[very thick,red] (4.6,0) ellipse (0.4 and 0.3);
\draw[very thick] (6.0,0) ellipse (0.6 and 0.3);
\draw[very thick] (7.8,0) ellipse (0.6 and 0.3);
\draw[very thick] (0,0) arc (180:0:4.2 and 2);
\draw[very thick] (3,0) arc (180:0:0.6 and 0.6);
\node at (1.5,0) {...};
\node at (3.4,0) {$S_A$};
\node at (4.6,0) {$S_B$};
\node at (6.9,0) {...};
\node at (4.2,2.3) {$a$};
\node at (3.6,0.9) {$e$};
\node at (0.6,-0.6) {$S_1$};
\node at (2.4,-0.6) {$S_{i-1}$};
\node at (6.0,-0.6) {$S_{i+1}$};
\node at (7.8,-0.6) {$S_{k}$};
\end{tikzpicture}
\end{minipage}

\begin{minipage}{0.04\textwidth}
\vphantom{a}    
\end{minipage}
\begin{minipage}{0.45\textwidth}
\centering
(\#0)  
\end{minipage}
\begin{minipage}{0.04\textwidth}
\vphantom{a}    
\end{minipage}
\begin{minipage}{0.45\textwidth}
\centering
(\#1)  
\end{minipage}
\vspace{5pt}

\begin{minipage}{0.07\textwidth}
\begin{tikzpicture}
\path [->>>,very thick] (0,0) edge node [midway, above] {\large $R$} (1.0,0);
\path [->,draw=none] (0,-1.8) (1.0,-1.8);
\end{tikzpicture}
\end{minipage}
\hspace{-0.03\textwidth}
\begin{minipage}{0.45\textwidth}
\begin{tikzpicture}[xscale=0.85,yscale=0.85]
\draw[very thick] (0.6,0) ellipse (0.6 and 0.3);
\draw[very thick] (2.4,0) ellipse (0.6 and 0.3);
\draw[very thick,red] (3.8,0) ellipse (0.4 and 0.3);
\draw[very thick,red] (4.6,0) ellipse (0.4 and 0.3);
\draw[very thick] (6.0,0) ellipse (0.6 and 0.3);
\draw[very thick] (7.8,0) ellipse (0.6 and 0.3);
\draw[very thick] (0,0) arc (180:0:4.2 and 2);
\draw[very thick] (3,0) arc (180:0:0.6 and 0.6);
\node at (1.5,0) {...};
\node at (3.8,0) {$S_A$};
\node at (4.6,0) {$S_B$};
\node at (6.9,0) {...};
\node at (4.2,2.3) {$a$};
\node at (3.6,0.9) {$e$};
\node at (0.6,-0.6) {$S_1$};
\node at (2.4,-0.6) {$S_{i-1}$};
\node at (6.0,-0.6) {$S_{i+1}$};
\node at (7.8,-0.6) {$S_{k}$};
\end{tikzpicture}
\end{minipage}
\hspace{-0.03\textwidth}
\begin{minipage}{0.07\textwidth}
\begin{tikzpicture}
\path [->,very thick] (0,0) edge node [midway, above] {\large $M_6$} (1.0,0);
\path [->,draw=none] (0,-1.8) (1.0,-1.8);
\end{tikzpicture}
\end{minipage}
\hspace{-0.03\textwidth}
\begin{minipage}{0.45\textwidth}
\begin{tikzpicture}[xscale=0.85,yscale=0.85]
\draw[very thick] (0.6,0) ellipse (0.6 and 0.3);
\draw[very thick] (2.4,0) ellipse (0.6 and 0.3);
\draw[very thick,red] (3.8,0) ellipse (0.4 and 0.3);
\draw[very thick,red] (4.6,0) ellipse (0.4 and 0.3);
\draw[very thick] (6.0,0) ellipse (0.6 and 0.3);
\draw[very thick] (7.8,0) ellipse (0.6 and 0.3);
\draw[very thick] (0,0) arc (180:0:4.2 and 2);
\draw[very thick] (4.2,0) arc (180:0:0.6 and 0.6);
\node at (1.5,0) {...};
\node at (3.8,0) {$S_A$};
\node at (4.6,0) {$S_B$};
\node at (6.9,0) {...};
\node at (4.2,2.3) {$a$};
\node at (4.8,0.9) {$e$};
\node at (0.6,-0.6) {$S_1$};
\node at (2.4,-0.6) {$S_{i-1}$};
\node at (6.0,-0.6) {$S_{i+1}$};
\node at (7.8,-0.6) {$S_{k}$};
\end{tikzpicture}
\end{minipage}

\begin{minipage}{0.04\textwidth}
\vphantom{a}    
\end{minipage}
\begin{minipage}{0.45\textwidth}
\centering
(\#2)  
\end{minipage}
\begin{minipage}{0.04\textwidth}
\vphantom{a}    
\end{minipage}
\begin{minipage}{0.45\textwidth}
\centering
(\#3)  
\end{minipage}
\vspace{5pt}

\begin{minipage}{0.07\textwidth}
\begin{tikzpicture}
\path [->>>,very thick] (0,0) edge node [midway, above] {\large $R$} (1.0,0);
\path [->,draw=none] (0,-1.8) (1.0,-1.8);
\end{tikzpicture}
\end{minipage}
\hspace{-0.03\textwidth}
\begin{minipage}{0.45\textwidth}
\begin{tikzpicture}[xscale=0.85,yscale=0.85]
\draw[very thick] (0.6,0) ellipse (0.6 and 0.3);
\draw[very thick] (2.4,0) ellipse (0.6 and 0.3);
\draw[very thick,red] (3.8,0) ellipse (0.4 and 0.3);
\draw[very thick,red] (5.0,0) ellipse (0.4 and 0.3);
\draw[very thick] (6.0,0) ellipse (0.6 and 0.3);
\draw[very thick] (7.8,0) ellipse (0.6 and 0.3);
\draw[very thick] (0,0) arc (180:0:4.2 and 2);
\draw[very thick] (4.2,0) arc (180:0:0.6 and 0.6);
\node at (1.5,0) {...};
\node at (3.8,0) {$S_A$};
\node at (5.0,0) {$S_B$};
\node at (6.9,0) {...};
\node at (4.2,2.3) {$a$};
\node at (4.8,0.9) {$e$};
\node at (0.6,-0.6) {$S_1$};
\node at (2.4,-0.6) {$S_{i-1}$};
\node at (6.0,-0.6) {$S_{i+1}$};
\node at (7.8,-0.6) {$S_{k}$};
\end{tikzpicture}
\end{minipage}
\hspace{-0.03\textwidth}
\begin{minipage}{0.07\textwidth}
\begin{tikzpicture}
\path [->,very thick] (0,0) edge node [midway, above] {\large $M_7$} (1.0,0);
\path [->,draw=none] (0,-1.8) (1.0,-1.8);
\end{tikzpicture}
\end{minipage}
\hspace{-0.03\textwidth}
\begin{minipage}{0.45\textwidth}
\begin{tikzpicture}[xscale=0.85,yscale=0.85]
\draw[very thick] (0.6,0) ellipse (0.6 and 0.3);
\draw[very thick] (2.4,0) ellipse (0.6 and 0.3);
\draw[very thick,red] (3.8,0) ellipse (0.4 and 0.3);
\draw[very thick,red] (5.0,0) ellipse (0.4 and 0.3);
\draw[very thick] (6.0,0) ellipse (0.6 and 0.3);
\draw[very thick] (7.8,0) ellipse (0.6 and 0.3);
\draw[very thick] (0,0) arc (180:0:4.2 and 2);
\draw[very thick] (3.4,0) arc (180:0:1 and 0.8);
\node at (1.5,0) {...};
\node at (3.8,0) {$S_A$};
\node at (5.0,0) {$S_B$};
\node at (4.4,-0.6) {$S_i$};
\node at (6.9,0) {...};
\node at (4.2,2.3) {$a$};
\node at (4.4,1.1) {$e$};
\node at (0.6,-0.6) {$S_1$};
\node at (2.4,-0.6) {$S_{i-1}$};
\node at (6.0,-0.6) {$S_{i+1}$};
\node at (7.8,-0.6) {$S_{k}$};
\end{tikzpicture}
\end{minipage}

\begin{minipage}{0.04\textwidth}
\vphantom{a}    
\end{minipage}
\begin{minipage}{0.45\textwidth}
\centering
(\#4)  
\end{minipage}
\begin{minipage}{0.04\textwidth}
\vphantom{a}    
\end{minipage}
\begin{minipage}{0.45\textwidth}
\centering
(\#5)  
\end{minipage}

\caption{An example of a shift sequence. 
The black longer bubbles represent minimal segments, and there may be $0$ or more of them. 
The red shorter bubbles represent valid sub-trees and are not necessarily minimal. }
\label{fig:shift-recursion}
\end{figure}
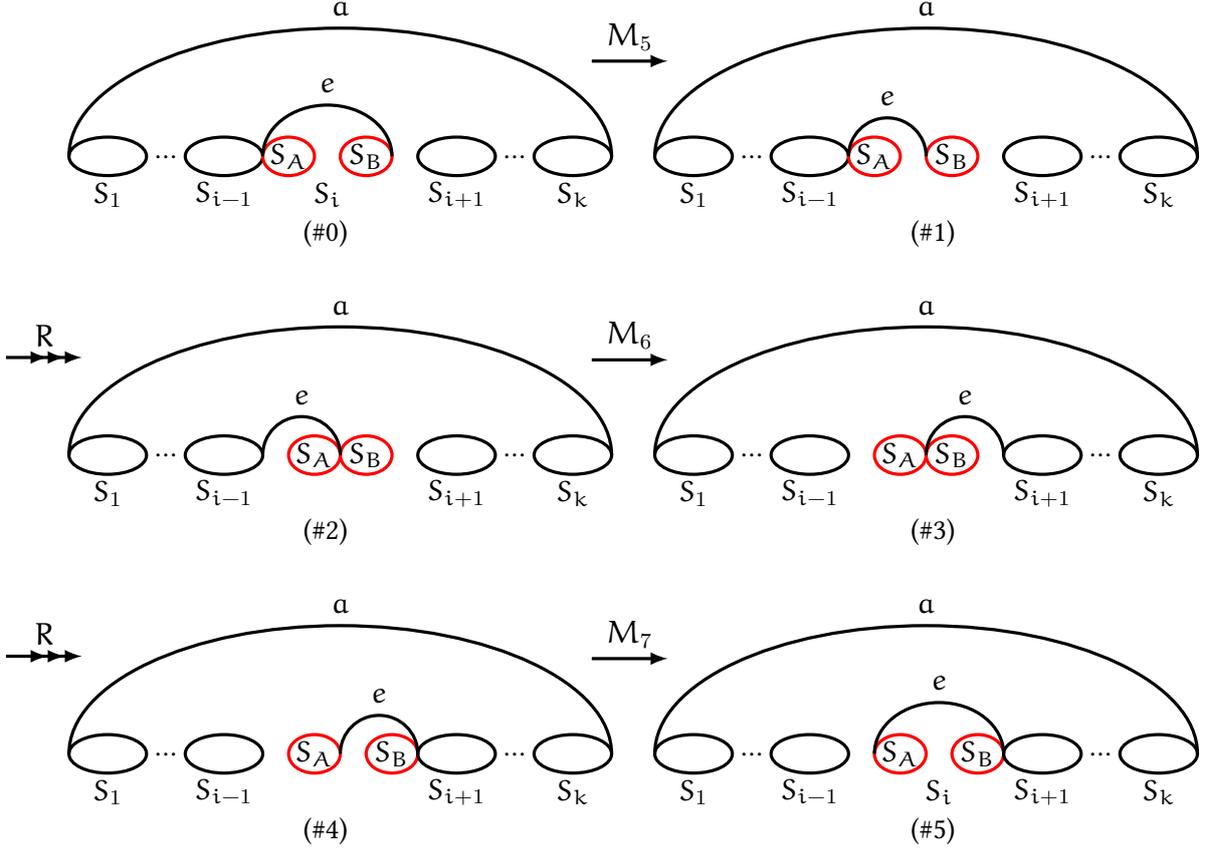

\begin{lemma} \label{lem:AM-shifting}
Let $T$ be a non-crossing spanning tree of $m$ edges decomposing into $(T_A,T_B,T_C)$ by \Cref{fac:ncst-decompose}, where $T_B$ and $T_C$ are empty, and $T_A$ is not empty. 
Then we can use $\ell\leq 3m$ flip moves to simulate a right shifting, i.e., reach a non-crossing spanning tree $T'$ decomposing into $(T'_A,T'_B,T'_C)$ where both $T'_A$ and $T'_C$ are empty, and $T'_B=T_A$. 
\end{lemma}

To prove this, we design a recursive way of simulating shifts.
We decompose the subtree to shift into minimal segments, and then move them one at a time.
To move a minimal segment, we perform edge flips so that subtrees of the minimal segment are in the form of a shift,
and finish moving it using recursion.
See \Cref{fig:shift-recursion} for an illustration of an example.

\begin{proof}
We construct the sequence by induction on the number of edges of the tree. 
Let $T$ be a non-crossing spanning tree where $T_B$ and $T_C$ are empty. 
Denote by $a=(a_1,a_2)$ the pivot edge of $T$. 
Let $S_1,\cdots,S_k$ be the minimal segments of $T_A$ from left to right. 
We perform the following one by one for each minimal segment, such that after finishing the first $t$ minimal segments, we obtain a non-crossing spanning tree $T''$ such that $T''_A$ contains $S_1,\cdots,S_{k-t}$, and $T''_B$ contains $S_{k-t+1},\cdots,S_{k}$

For each current segment $S=S_i$, let $e$ be its pivot edge, and we perform a sequence of three flip moves, $M_5,M_6,M_7$, with recursively constructed shift sequences under $e$ after $M_5$ and $M_6$, respectively,
to shift the subtrees of $S$ one by one. 
Note that the overarching edge of $e$ is $a$. 
Decompose $S$ into $(S_A, S_B, \emptyset)$; the last component is empty because $S$ is a minimal segment. 
Let $p$ be the initial left endpoint of its pivot edge $e$, and fix $g_e$ to be the start of the gap before any flip moves have been made beneath $e$. 
Let $g_a$ be the start of the current gap beneath $a$. 
Therefore, the initial position of $e$ is $(p,g_a)$, 
$S_A$ is given by $[p,g_e]$, 
and $S_B$ is given by $[g_e+1,g_a]$.

First assume that neither of $S_A$ or $S_B$ is empty. The simulation of the shifting is given by the following. 
\begin{itemize}
\item[($M_5$)] Let $M_5$ move $e$ from $(p,g_a)$ to $(p,g_e+1)$. 
Since the left endpoint of $e$ remains connected to the subtree $S_B=[g_e+1,g_a]$, the result is still a tree, and since there are no edges other than $e$ going from $S_A$ to $S_B$, the result is still non-crossing, and thus $M_5$ is a valid flip move. 
\item[($R$)] Now the position of $S_A$ and the edge $e$ meet the prerequisite for a right shifting. We perform a right shifting. 
\item[($M_6$)] Let $M_6$ move $e$ from $(p,g_e+1)$ to $(g_e+1,g_a+1)$. 
The resulting graph is still a tree, because one endpoint of $e$ that joins $S_A$ and $S_B$ remains unchanged, and the other endpoint is still outside the $S_AS_B$ part after the move. 
Moreover, since the overarching edge of $e$ is $a$, there is no edge connecting any point in $S_B$ with another point to the left of $p$. And $g_a$ is the gap beneath the edge $a$, meaning there is no edge joining any point in $S_B$ with another point to the right of $g_a+1$. 
This means the resulting tree is still non-crossing,
and hence $M_6$ is a valid flip move. 
\item[($R$)] Now the position of $S_B$ and the edge $e$ meet the prerequisite for a right shifting. We perform a right shifting. 
\item[($M_7$)] Let $M_7$ move $e$ from $(g_e+1,g_a+1)$ to $(p+1,g_a+1)$. 
A similar argument as $M_5$ also gives that $M_7$ is a valid flip move. 
\end{itemize}

We tweak the above procedure a bit when at least one of $S_A$ and $S_B$ is empty:
\begin{enumerate}
\item When both subtrees are empty, we only use $M_6$. 
Note that since $S_A$ is empty, we meet the assumption for $M_6$, and since $p=g_e=g_a-1$, $e$ is in the correct initial position for $M_6$ and $S$ has shifted right. This is our base case. 
\item When $S_A$ is empty but $S_B$ is not, we use moves $M_5$ and $M_6$, and then recursively construct a shift sequence under $e$. 
Since $S_A$ is empty, we satisfy the assumption for $M_6$. 
Note $p=g_e$, so $M_6$ puts $e$ in a spot shifted right from its initial position, and by the induction $S_A$ is shifted right as well, so $S$ has shifted right. 
\item When $S_A$ is not empty but $S_B$ is, we recursively construct a shift sequence under $e$, and then use $M_6$ and $M_7$. 
Since $S_B$ is empty and we begin by shifting $A$, we satisfy the assumptions for $M_6$ and $M_7$, and since $g_e+1=g_a$, $e$ is in the correct initial position for $M_6$. 
Note $M_7$ puts $e$ in a spot shifted right from its initial position, and by the induction, we have shifted $S_A$ right, so $S$ has shifted right. 

\end{enumerate}

We concatenate the sequences for each $S_i$ to generate a sequence to shift $T_A$ right. 
Note that each edge except the pivot edge of $T$ is moved one to three times, so the total length of the shift sequence is at most $3m$. 
\end{proof}

\subsection{Path construction}

We are going to apply the path method, namely \Cref{thm:path-method}, to compare $P_{\mathrm{AM}}$ and $P_{\mathrm{FM}}$. 
Because of the bijection between $2$-Dyck paths and NCSTs,
we simply treat $P_{\mathrm{AM}}$ and $P_{\mathrm{FM}}$ as having the same stationary distribution.
To compare the two chains, we will need to construct a path (namely, a sequence of transitions) for any two $2$-Dyck paths $I$ and $F$ where $P_{\mathrm{AM}}(I,F)>0$ using only transitions in $P_{\mathrm{FM}}$.
We have already described how to simulate an adjacent move via flip moves and shifting in \Cref{lem:AM-simulation}, and how to simulate shifting via flip moves in \Cref{lem:AM-shifting}.
We summarise the whole construction in \Cref{alg:canonical-path-ncst}, where we expand out the Type-1 move and omit the details of the Type-2 move to avoid redundancy. 
(The flip moves are already described in \Cref{lem:AM-simulation} and \Cref{fig:FM-sequences}(2$\ast$), and the \textsf{ShiftLeft} sequence is the reverse of \textsf{ShiftRight}.)

Let $T_I$ and $T_F$ be the non-crossing spanning trees corresponding to $I$ and $F$. 
The construction yields a path $T_I=Z_0\to Z_1\to\cdots\to Z_\ell=T_F$ where $P_{\mathrm{FM}}(Z_i,Z_{i+1})>0$ for any $i\le \ell-1$. 
Its validity, namely that each transition is valid, is implied by \Cref{lem:AM-simulation} and \Cref{lem:AM-shifting}. 
These two lemmas also imply an upper bound on the length of the path $\ell\leq 3n+2$ ($3n$ for the shift sequence and $2$ for the two simple flips).

\begin{algorithm}[p]
\caption{Constructing the canonical path}\label{alg:canonical-path-ncst}
\SetKwProg{Fn}{procedure}{\string:}{}
\KwIn{two $2$-Dyck paths $I\neq F$ of length $3n$ that $P_{\mathrm{AM}}(I,F)>0$}
\KwOut{a canonical path $I=Z_0, Z_1, \cdots, Z_{\ell}$}
Decide if $I$ to $F$ is a Type-1 or Type-2 move as in \Cref{lem:AM-simulation}\;
\eIf{It is a Type-1 move}{
\Return{\normalfont\textsf{Transition1}($I,F$)}\;
}{
\Return{\normalfont\textsf{Transition2}($I,F$)}\;
}
\BlankLine
\Fn{\normalfont\textsf{Transition1}($I,F$)}{
\tcp{It is required that $I\to F$ is a Type-1 move.}
Let $Z_0\gets I$\;
Find $a=(a_1,a_2)$, $e=(p,g_e+1)$ and $g_a$ as in \Cref{lem:AM-simulation}\;
\If{$p\leq g_e-1$}{
  Let $(\ell-2,Z_1,\cdots,Z_{\ell-2})\gets$\textsf{ShiftRight}($p,g_e+1,1,1$)\;
}
Perform an $M_1$ move of $e$ to obtain $Z_{\ell-1}$\;
Let $Z_\ell\gets F$, which is obtained by an $M_2$ move of $e$ from $Z_{\ell-1}$\;
\Return{$Z_0,\cdots,Z_\ell$}\;
}
\BlankLine
\Fn{\normalfont\textsf{ShiftRight}($s,t,\mathsf{depth},\mathsf{step}$)}{
Let $k\geq 1$ be the number of minimal segments $S_1,\cdots,S_k$ beneath the edge $(s,t)$\;
$\mathsf{last}\gets\mathsf{step}$\;
\For{$i$ from $k$ downto $1$}{
    Let $e=(p,q)$ be the pivot edge of $S_i$, and $g_e,g_{e+1}$ be the gap beneath $e$\tcp*{(\#0)}
    \If{$q\geq g_e+2$}{
        Perform an $M_5$ move of $e$ to obtain $Z_\mathsf{last}$, and update $\mathsf{last}\gets \mathsf{last}+1$\tcp*{(\#1)}
    }
    \If{$p\leq g_e-1$}{
      $(\ell',Z_{\mathsf{last}},\cdots,Z_{\mathsf{last}+\ell'-1})\gets$\textsf{ShiftRight}($p,g_e+1,\mathsf{depth}+1,\mathsf{last}$)\;
        Update $\mathsf{last}\gets \mathsf{last}+\ell'$\tcp*{(\#2)}
    }
    Perform an $M_6$ move of $e$ to obtain $Z_\mathsf{last}$, and update $\mathsf{last}\gets \mathsf{last}+1$\tcp*{(\#3)}
    \If{$q\geq g_e+2$}{
      $(\ell'',Z_{\mathsf{last}},\cdots,Z_{\mathsf{last}+\ell''-1})\gets$\textsf{ShiftRight}($g_e+1,q+1,\mathsf{depth}+1,\mathsf{last}$)\;
        Update $\mathsf{last}\gets \mathsf{last}+\ell''$\tcp*{(\#4)}
    }
    \If{$p\leq g_e-1$}{
        Perform an $M_7$ move of $e$ to obtain $Z_\mathsf{last}$, and update $\mathsf{last}\gets \mathsf{last}+1$\tcp*{(\#5)}
    }
}
\Return{$(\mathsf{last}-\mathsf{step},Z_{\mathsf{step}},\cdots,Z_{\mathsf{last}})$}\;
}
\BlankLine
\Fn{\normalfont\textsf{Transition2}($I,F$)}{
\tcp{It is required that $I\to F$ is a Type-2 move.}
\tcp{Omitted}
}
\Fn{\normalfont\textsf{ShiftLeft}($s,t,\mathsf{depth},\mathsf{step}$)}{
\tcp{Omitted}
}
\end{algorithm}
\afterpage{\clearpage}

\subsection{Bound the congestion}
The goal of the Markov chain comparison argument is to show the following lemma:
\begin{lemma} \label{lem:AM-vs-FM}
Let $\mathcal{E}_{\mathrm{AM}}$ be the Dirichlet form of the adjacent move chain $P_{\mathrm{AM}}$ over the uniform distribution of $2$-Dyck paths of length $3n$, 
and $\mathcal{E}_{\mathrm{FM}}$ be the Dirichlet form of the flip chain $P_{\mathrm{FM}}$ over the uniform distribution of non-crossing spanning trees containing $n$ edges. 
Then for all functions $f:\mathcal{B}\rightarrow\mathbb{R}$,\footnote{With a slight notation abuse, we assume $f(X)=f(T_X)$ due to the one-to-one correspondence.} 
\[\mathcal{E}_{\mathrm{FM}}(f,f)\geq\Omega(n^{-4})\cdot\mathcal{E}_{\mathrm{AM}}(f,f).\]
\end{lemma}

The next lemma shows that our canonical path construction ensures that each transition in the flip move $P_{\mathrm{FM}}$ is used by a limited number of $(I,F)$ pairs in the adjacent move $P_{\mathrm{AM}}$. 
This is proved by designing an encoding that we can uniquely recover the $(I,F)$ pair by looking at the current transition and the encoding. 
The number of $(I,F)$ pairs is then bounded by the number of possible encodings, which we ensure to be linear in $n$. 

\begin{lemma}[encoding] \label{lem:AM-vs-FM-encoding}
Let $Z\to Z'$ be a transition of $P_{\mathrm{FM}}$ over non-crossing spanning trees containing $n$ edges. 
Let $I\neq F$ be 
two $2$-Dyck paths
such that $P_{\mathrm{AM}}(I,F)>0$, and that the path constructed according to \Cref{alg:canonical-path-ncst} contains the transition $Z\to Z'$. 
Then there exists an injective mapping
\[
\varphi_{Z\to Z'}: \mathcal{T}_n\times\mathcal{T}_n\to\{\mathsf{Left},\mathsf{Right}\}\times\{M_1,M_2,M_3,M_4,S_1,S_2\}\times [n]
\]
where $\mathcal{T}_n$ is the set of all non-crossing spanning trees of $n$ edges. 
In other words, given $Z\to Z'$, and $(\mathsf{dir},M,d)$, we can uniquely determine $I,F$ such that $\varphi_{Z\to Z'}(T_I,T_F)=(\mathsf{dir},M,d)$,
where $\mathsf{dir}\in\{\mathsf{Left},\mathsf{Right}\}$, $M\in\{M_1,M_2,M_3,M_4,S_1,S_2\}$, and $d\in[n]$. 
\end{lemma}

Let $\Gamma_{IF}$ be the path from $I$ to $F$.
We write $(Z,Z')\in \Gamma_{IF}$ if the $(Z,Z')$ transition is in the path $\Gamma_{IF}$.
\Cref{lem:AM-vs-FM-encoding} immediately implies the following.

\begin{corollary} \label{cor:FM-congestion}
  For each $(Z,Z')$ with $P_{\mathrm{FM}}(Z,Z')>0$, there are at most $12n$ pairs $(I,F)$ such that $(Z,Z')\in \Gamma_{IF}$. 
\end{corollary}

We first conclude the proof of \Cref{lem:AM-vs-FM} by using the above encoding. 

\begin{proof}[Proof of \Cref{lem:AM-vs-FM}]
  Since $P_{\mathrm{AM}}$ and $P_{\mathrm{FM}}$ are reversible, \Cref{thm:path-method} applies. 
  The key quantity, namely the congestion, is bounded as follows:
  \begin{align*}
    B&=\max_{\substack{(Z,Z'):\\P_\mathrm{FM}(Z,Z')>0}}\sum_{\substack{(I,F):\\(Z,Z')\in \Gamma_{IF}}}\frac{\ell_{IF}\cdot P_\textrm{AM}(I,F)}{P_\textrm{FM}(Z,Z')}
    \leq\max_{\substack{(Z,Z'):\\P_\mathrm{FM}(Z,Z')>0}}\sum_{\substack{(I,F):\\(Z,Z')\in \Gamma_{IF}}}\frac{3n+2}{(6n-2)P_\textrm{FM}(Z,Z')}\\
    &\leq\frac{5}{4}\max_{\substack{(Z,Z'):\\P_\mathrm{FM}(Z,Z')>0}}\frac{1}{P_\mathrm{FM}(Z,Z')}\abs{\{(I,F):(Z,Z')\in \Gamma_{IF}\}}=O(n^4),
  \end{align*}
  where the last bound is due to \eqref{eqn:fm-transition-bound} and \Cref{cor:FM-congestion}. 
\end{proof}

%
%
%
%

\begin{proof}[Proof of \Cref{lem:AM-vs-FM-encoding}]
Recall that the edge flips defined in \Cref{lem:AM-simulation}, which we shall refer to as \emph{the top level}, invoke once the subroutine of constructing the shift sequence if the part to shift is non-empty. 
The shift sequence construction subroutine \textsf{ShiftLeft} or \textsf{ShiftRight} (\Cref{lem:AM-shifting}) might make further recursive calls. 
This inspires us to define the mapping
\[
\varphi_{Z\to Z'}(T_I,T_F)=(\mathsf{dir},M,d),
\]
where
\begin{itemize}
\item $\mathsf{dir}\in\{\mathsf{Left},\mathsf{Right}\}$ indicates the direction of adjacent move from $I$ to $F$, 
\item $M\in\{M_1,M_2,M_3,M_4,S_1,S_2\}$ indicates where the current transition $Z\to Z'$ is constructed as in \Cref{lem:AM-simulation}, and
\item $d\in[n]$ indicates the depth of the recursion of the shift sequence construction. 
Formally, it is the parameter $\mathsf{depth}$ of the recursive call where $Z'$ is obtained in \textsf{ShiftRight} (or \textsf{ShiftLeft}) in \Cref{alg:canonical-path-ncst}. 
It is defined only when $M\in\{S_1,S_2\}$, and otherwise arbitrary. 
\end{itemize}
The mapping defined above is total because the level of recursion cannot go beyond the number of edges $n$. 
It is then left for us to show that it is injective. 

Without loss of generality, we only consider the case $\mathsf{dir}=\mathsf{Left}$; otherwise, we can just swap $Z$ and $Z'$ to obtain the unique reconstruction of $I$ and $F$ using the argument for the case $\mathsf{dir}=\mathsf{Left}$, and then swap $I$ and $F$. 

If $M\in\{M_1,M_2,M_3,M_4\}$, the initial state $I$ and the final state $F$ can be very easily determined. 
Consider $M=M_1$; the other cases have a similar argument. 
This case corresponds to the \textbf{Type 1} move. 
We can find the edge being moved by looking at $Z\to Z'$, which is the edge $e$ as in the statement of $M_1$ move in \Cref{lem:AM-simulation}. 
We can also determine the edge $a$ therein because it is the overarching edge of $e$. 
The initial $I$ is determined by shifting the subtree beneath $e$ in $Z$ back to the left, and the final $F$ is determined by performing $M_2$ from $Z'$ (as we already know what the two edges are). 

Otherwise, $M=S_1$, which is a right shifting in the \textbf{Type 1} move, or $M=S_2$, which is a left shifting in the \textbf{Type 2} move. 
We prove the former case, as the argument for the latter one is identical. 


Let $e_d$ be the edge being moved when the move $Z\to Z'$ is performed. 
From $i=d-1$ to $i=0$, we can determine the overarching edge of $e_{i+1}$ from $Z$ (or $Z'$), denoted by $e_i$, whose current two endpoints in $Z$ (or $Z'$) are $(s_i,t_i)$.
Finally, find the overarching edge $a$ of $e_0$. 
The recursive calls that involve the current edge $e_d$ are therefore $\mathsf{ShiftRight}(s_i,t_i,i+1,\mathsf{step}_i)$ for $i=0,1,\cdots,d-1$.
We recover the initial on each level from bottom to top to obtain $Z_{\mathsf{step}_{d-1}-1},Z_{\mathsf{step}_{d-2}-1},\cdots,Z_{\mathsf{step}_{0}-1}=Z_1$ as follows. 

At recursion level $d$, the transition $Z\to Z'$ must be an edge move $M_5$, $M_6$ or $M_7$ at this level. 
This can be determined by the following argument. (See \Cref{fig:shift-recursion} as a reference.)
\begin{itemize}
\item If $e_d$ in $Z$ starts left before the gap beneath $e_{d-1}$, and if the left end of $e_d$ remains unchanged in $Z'$, then it is an $M_5$ move with a non-empty $S_B$ part. 
\item Otherwise, if $e_d$ in $Z'$ ends right after the gap beneath $e_{d-1}$, and if the right end of $e_d$ is the same in $Z$ and $Z'$, then it is an $M_7$ move with a non-empty $S_A$ part. 
\item Otherwise, it is an $M_6$ move. 
\end{itemize}
Knowing which move it is allows us to determine the $S_A$ and $S_B$ part of the current minimal segment $S_i$ that is being processed. 
To recover the initial state beneath $e_{d-1}$, we first recover the current minimal segment $S_i$ (obtaining \Cref{fig:shift-recursion}(\#0)), 
and then move everything beneath $e_{d-1}$ after the gap beneath $e_{d-1}$ to the left by one position. 
In this way, we revert the subroutine $\mathsf{ShiftRight}(s_{d-1},t_{d-1},d,\mathsf{step}_{d-1})$ and obtain $Z_{\mathsf{step}_{d-1}-1}$. 

Now suppose that we are at recursion depth $d'<d$. 
The current transition $Z\to Z'$ must be inside a recursive call on this level, which means $Z$ and $Z'$ are the same except the part with coordinates from $s_{d'+1}$ to $t_{d'+1}$ (namely the part beneath $e_{d'}$). 
There are two possible recursive calls in $\mathsf{ShiftRight}$. 
To determine which kind of recursive call it is (either the first one, between (\#1) and (\#2) in \Cref{alg:canonical-path-ncst}, or the second one, between (\#3) and (\#4)), we look at the gap beneath $e_{d'-1}$. 
If it is on the right of $e_{d'}$, then the recursive call is the first one, or otherwise it is the second one. 
In any case, we first recover the part beneath $e_{d'}$ and obtain $Z_{\mathsf{step}_{d'}-1}$.
At the current level, this ends up in either the state (\#1) or (\#3), which further allows us to go back to state (\#0) by recovering the initial state of the current minimal segment, 
and then move everything beneath $e_{d'-1}$ after the gap beneath $e_{d'-1}$ to the left by one position.  
In this way, we revert the subroutine $\mathsf{ShiftRight}(s_{d'-1},t_{d'-1},d',\mathsf{step}_{d'-1})$ and obtain $Z_{\mathsf{step}_{d'-1}-1}$. 

The above induction provides us a way to reconstruct the initial state beneath $e_0$, which gives $Z_0$. 

To recover the final state $T_F$, one only needs to modify the above procedure a bit, by moving every minimal segments to the right in each call of $\mathsf{ShiftRight}$, and perform the $M_1$ and $M_2$ move on $e_0$ and $a$ on the top level. 

The encoding is unique because, if there are two pairs $(I,F),(I',F')$ such that they receive the same encoding $\varphi_{Z\to Z'}(T_I,T_F)$, the above procedure correctly recovers the initial and final states. 
Because the procedure is deterministic and only depends on $Z\to Z'$ and the encoding, it generates the same output as containing the initial and final states, which means $(I,F)=(I',F')$.
\end{proof}

We end this section by concluding the proof of the main theorem. 

\begin{proof}[Proof of \Cref{thm:NCST}]
We combine our comparison of the adjacent move and flip chains to get
\begin{align*}
t_\mathrm{mix}(P_\mathrm{FM})&\leq\left(1+\frac{1}{2}\log\pi_\mathrm{min}\right)\frac{1}{\lambda(P_\mathrm{FM})}\\
&=\left(1+\frac{1}{2}\log\pi_\mathrm{min}\right)\frac{1}{\inf\left\{\frac{\mathcal{E}_\mathrm{FM}(f,f)}{\Var{\pi}{f}}\mid f:\Omega\to\mathbb{R},\Var{\pi}{f}\neq0\right\}}\tag*{(by \Cref{equ:lambda-def})}\\
&\leq\left(1+\frac{1}{2}\log\pi_\mathrm{min}\right)\frac{O(n^4)}{\inf\left\{\frac{\mathcal{E}_\mathrm{AM}(f,f)}{\Var{\pi}{f}}\mid f:\Omega\to\mathbb{R},\Var{\pi}{f}\neq0\right\}}\tag*{(by \Cref{lem:AM-vs-FM})}\\
&=\left(1+\frac{1}{2}\log\pi_\mathrm{min}\right)\frac{O(n^4)}{\lambda(P_\mathrm{AM})}\tag*{(by \Cref{equ:lambda-def})}\\
&=\left(1+\frac{1}{2}\log\pi_\mathrm{min}\right) O(n^7\log n)\tag*{(by \Cref{cor:wilson-spectral-gap})}\\
&= O(n^8\log n). \qedhere
\end{align*}
\end{proof}

\bibliographystyle{alpha}
\bibliography{refs}


\end{document}